\documentclass[11pt,a4paper]{amsart}
\usepackage[english]{babel}
\usepackage{pst-grad} 
\usepackage{pst-plot} 
\usepackage{pstricks}
\usepackage{amsmath,amssymb,mathrsfs,amsthm}
\usepackage[dvips]{graphicx}
\usepackage{epsfig}

\newcommand{\RR}{\mathbb R}
\newcommand{\Ss}{\mathbb S}

\newcommand{\NN}{\mathbb N}
\newcommand{\ZZ}{\mathbb Z}

\newcommand{\TT}{\mathbb T}
\newcommand{\pat}{\partial_t}
\newcommand{\pax}{\partial_x}

\newcommand{\jeps}{\mathcal{J}_\epsilon*}

\newcommand{\vertiii}[1]{{\left\vert\kern-0.25ex\left\vert\kern-0.25ex\left\vert #1 
    \right\vert\kern-0.25ex\right\vert\kern-0.25ex\right\vert}}

\newcommand{\re}{\text{Re}}
\newcommand{\im}{\text{Im}}

\usepackage[pdfauthor={Rafael Granero-Belinchon},pdftitle={...},bookmarks,colorlinks]{hyperref}

\newcounter{comentcount}
\newcounter{teocount}
\newtheorem{lem}{Lemma}

\newtheorem{corol}{Corollary}
\newtheorem{teo}[teocount]{Theorem}  
\newtheorem{defi}{Definition}

\newtheorem{remark}{Remark}

\title[Generalized Keller-Segel]{On a generalized doubly parabolic Keller-Segel system in one spatial dimension}

\author[J. Burczak]{Jan Burczak}
\email{jb@impan.pl}
\address{Institute of Mathematics of the Polish Academy of Sciences, Warsaw, 21 00-956, Poland}

\author[R. Granero-Belinch\'{o}n]{Rafael Granero-Belinch\'{o}n}
\email{rgranero@math.ucdavis.edu}
\address{Department of Mathematics, University of California, Davis, CA 95616, USA}

\begin{document}
\begin{abstract}We study a doubly parabolic Keller-Segel system in one spatial dimension, with diffusions given by fractional laplacians. We obtain several local and global well-posedness results for the subcritical and critical cases (for the latter we need certain smallness assumptions). We also study dynamical properties of the system with added logistic term. Then, this model exhibits a spatio-temporal chaotic behavior, where a number of peaks emerge. In particular, we prove the existence of an attractor and provide an upper bound on the number of peaks that the solution may develop. Finally, we perform a numerical analysis suggesting that there is a finite time blow up if the diffusion is weak enough, even in presence of a damping logistic term. Our results generalize on one hand the results for local diffusions, on the other the results for the parabolic-elliptic fractional case.
\end{abstract}



\maketitle


\section{Introduction}

This paper is devoted to studies of the following generalized, doubly parabolic ($\tau=1$) Keller-Segel-type system with a logistic term ($r \ge 0$)
\begin{eqnarray}\label{eqa1}
\pat u & = & -\mu\Lambda^\alpha u+\pax(u\Lambda^{\beta-1} H v) +ru(1-u),\\ 
\label{eqa2}
\tau \pat v & = & -\nu\Lambda^\beta v-\lambda v+u,
\end{eqnarray}
on $\TT$, \emph{i.e.} the one dimensional periodic torus, where  $\Lambda =\sqrt{-\Delta}$ (for basic notation and definitions, see Section \ref{ssub:prel}). A similar model has been mentioned by Biler \& Wu, see \cite{BilerWu}, Section 5. In \eqref{eqa1}-\eqref{eqa2} we take parameters $\nu,\mu,\alpha, \beta>0,\lambda, r\geq0$ and nonnegative initial data $u_0$ and $v_0$. We will refer to \eqref{eqa1}-\eqref{eqa2} with $\tau=0$ as to the parabolic-elliptic system and with $\tau=1$ as to the doubly parabolic one. In order to clarify the terminology, let us simply define the case $\alpha >1$ as subcritical,  $\alpha =1$ as critical and  $\alpha <1$ as supercritical. 

The system \eqref{eqa1}-\eqref{eqa2} with  $\tau = 0$ and $\alpha =1$, $ \beta = 2$ is the $\mu(u) \equiv \mu$ simplification of the one considered by us in \cite{BG}, \emph{i.e.} 
\begin{eqnarray*}
\pat u & =&  \pax(-\mu(u)Hu+u\pax v) +ru(1-u), \\ 
0 & =&-\pax^2 v +   u-\langle u \rangle.
\end{eqnarray*}
\subsection{Motivation}\label{ssec:mot}
\subsubsection{Mathematical biology}
Our interest in  the system \eqref{eqa1}-\eqref{eqa2} stems from the mathematical studies of chemotaxis initiated by Keller \& Segel in \cite{keller1970initiation}. Chemotaxis is a chemically prompted motion of cells with density $u$ towards increasing concentrations of a chemical substance with density $v$. For instance, in the case of  the slime mold \emph{Dictyostelium Discoideum}, the signal is produced by the cells themselves and cell populations might form aggregates in finite time. Chemotaxis also takes place in certain bacterial populations, such as of \emph{Escherichia coli} and \emph{Salmonella typhimurium}, and it results in their arrangement into a variety of spatial patterns. During embryogenesis, chemotaxis plays a role in angiogenesis, pigmentation patterning and neuronal development. It is also important in cancerogenesis, since certain tumors force the host organism to link them with its blood system via chemical signals. Specifically, in presence of the logistic term, our model is of particular importance in view of its relationship with the three-component urokinase plasminogen invasion model (see Hillen, Painter \& Winkler \cite{Hillen1}).

Moreover, let us observe that the cell kinetics model M$8$ in Hillen \& Painter \cite{Hillen3}, that describes a  bacterial pattern formation or cell movement and growth during angiogenesis, reads
\begin{eqnarray}\label{eqa1local}
\pat u & = & \mu \Delta u-\nabla(u\nabla v) +ru(1-u)\\ 
\label{eqa2local}
\pat v & = & \nu \Delta v-\lambda v+u.
\end{eqnarray}
System \eqref{eqa1local}-\eqref{eqa2local} in one dimension is especially close to our system  \eqref{eqa1}-\eqref{eqa2}, since it is  given by choosing $\alpha = \beta = 2$ in  \eqref{eqa1}-\eqref{eqa2} with $\tau =1$.

The parabolic-elliptic ($\tau=0$) version of the system \eqref{eqa1local}-\eqref{eqa2local}  is close to astrophysical models of a gravitational collapse. It is very similar in spirit to the Zel'dovich approximation \cite{zel1970gravitational} used in cosmology to study the formation of large-scale structures in the primordial universe, see also Ascasibar, Granero-Belinch\'on \& Moreno \cite{AGM}. It is also connected with the Chandrasekhar equation for the gravitational equilibrium of polytropic stars, statistical mechanics and the Debye system for electrolytes, see Biler \& Nadzieja \cite{BilNad94}.
A more detailed presentation of some results on systems of type  \eqref{eqa1local}-\eqref{eqa2local} and \eqref{eqa1}-\eqref{eqa2} follows in Section \ref{S9}. 
\subsubsection{Fractional diffusion}
The importance of the fractional diffusion generalization \eqref{eqa1}-\eqref{eqa2} of   \eqref{eqa1local}-\eqref{eqa2local} is twofold. 

Primarily, there is a serious mathematical interest involved. To explain this point, let us recall that chemotaxis systems model two opposite phenomena: one is diffusion of cells due to their random movements, the other is their tropism toward higher concentrations of a chemical that may result in their aggregations. Hence it is mathematically interesting to establish the minimal strength of diffusion that overweights the chemotactic forces, hence giving, roughly speaking, the global existence of regular solutions or, equivalently, to study the maximal strength of diffusion that does not prevent blowup. 

Let us recall that for the parabolic-elliptic in two space dimensions the standard diffusion $\Delta$ is critical; moreover the exact initial mass $\|u_0 \|_{L^1}$ that divides the regimes of global existence and of blowup has been computed, compare for instance Bournavas \& Calvez \cite{BouCalChapter} and its references. Let us remark here that the blowup phenomenon together with the mass threshold was shown by J\"ager \& Luckhaus \cite{JagLuc92} and Nagai \cite{Nagai95}.

For the doubly parabolic case in two space dimensions the situation is analogous, but here the available results are much later and less complete, see Mizoguchi \cite{mizoguchi13} and references therein.
In this context one may argue that the doubly parabolic case is substantially more difficult than the parabolic-elliptic one. We refer again to Section \ref{S9} for more detailed overview of the known results. 

In the one-dimensional case, the standard diffusion is strong enough to give the global existence; on the other hand, for $d >2$ it is too weak.  In this context it is mathematically interesting to find, for a fixed space dimension $d$, a \emph{critical} diffusive operator that sits on the borderline of the blowup and global-in-time regimes. There are at least two approaches to this problem, both justified from the point of view of applications. One is to consider the semilinear diffusion $\nabla \cdot (\mu ( u) \nabla u)$, see for instance Bedrossian, Rodriguez \& Bertozzi \cite{Bertozzi}, Blanchet, Carrillo \& Lauren\c{c}ot \cite{blanchet2009critical},  Cie\'slak \& Stinner \cite{CieslakStinner}, Burczak, Cie\'slak \& Morales-Rodrigo \cite{BurczakCieslak}, Cie\'slak \& Lauren\c{c}ot \cite{CieslakLaurencot} ad Tao \& Winkler \cite{TaoWinkler}. Another one is to replace the standard diffusion with the fractional one. In such a case there is a strong evidence that the half-laplacian $\Lambda =\sqrt{-\Delta}$ is especially worth studying; for more on this, see Subsection \ref{ssec:priorF}.

We focus on the latter approach and one dimension. 

Let us mention here that the logistic term generally helps the global existence, see  Tello \& Winkler \cite{TelloWinkler}, Winkler \cite{Winkler4}, Burczak \& Granero-Belinch\'on \cite{BG}. However, in view of our interest in large-time behavior of solutions to  \eqref{eqa1}-\eqref{eqa2}, we include the logistic term in our considerations here mainly due to the context in which it appears in \cite{Hillen4}, namely the spatio-temporal chaos.

Apart from the outlined mathematical interest in  fractional diffusion systems, it is also believed that they can be useful for modelling certain feeding strategies. For studies on microzooplancton, compare  Klafter, Lewandowsky \& White \cite{Klaf90} and Bartumeus, Peters, Pueyo, Marras{\'e} \& Catalan \cite{Bart03}; on amoebas -- Lewandowsky, White \& Schuster \cite{Lew_nencki}, flying ants -- Shlesinger \& Klafter \cite{Shl86}, fruit flies -- Cole \cite{Cole}, and jackals -- Atkinson, Rhodes, MacDonald \& Anderson \cite{Atk}.

Of course our system  \eqref{eqa1}-\eqref{eqa2} is merely motivated by the applications in biology and not directly applicable, since it concerns one space dimensions and no boundary. Nevertheless, let us remark here that most of the analysis in this paper can be carried out in the case where the domain is the real line. In particular, Theorems \ref{localexistence}, \ref{continuation}, \ref{globalwiener}, \ref{globalpp1}, \ref{globalpp1abs}, \ref{smoothingeffect} and also part of Theorem \ref{globalpp2} (when $\alpha>1$) can be adapted to the real line in a straightforward way.

\subsection{Plan of the paper and overview of our results}
In Section \ref{S9} we present in more details the known results on Keller-Segel-type systems. Section \ref{ssub:prel} introduces basic notation, function spaces and our notion of solution. Next, we provide precise statements of our main results as well as some additional remarks in Section \ref{S1}. 

The following sections contain proofs of our statements.

In particular, in Section \ref{S2} we prove local existence of solutions to \eqref{eqa1}-\eqref{eqa2}, while in Section \ref{S2b} we show continuation criteria. 
 
 Next, in Sections \ref{S10} and \ref{S3}, we study the global-in-time existence. More precisely, we prove global existence of regular solutions in the hypoviscous case $\alpha=1$, provided an explicit smallness condition for initial data holds and $\mu > 1$. This result holds for $r=0$.  Moreover, for arbitrary smooth initial data, we show global existence  of 
 \begin{itemize}
 \item weak solutions  in the subcritical case $\alpha>1$ for  $r \ge 0$ and in both critical and subcritical case $\alpha \ge 1$ for   $r > 0$,
 \item strong solutions in the subcritical case $\alpha>1$ with either $\beta\geq\alpha/2$ and $r \ge 0$ or with $\beta\geq0$ and $r > 0$.
 \end{itemize}
 
 Section \ref{pf:abs} provides results on the absorbing set in the case $\beta \ge \alpha >1$, $r>0$.
 
  In Section \ref{S5} we study the smoothing properties of the systems \eqref{eqa1}-\eqref{eqa2}, including an instantaneous gain of analyticity of the solution to \eqref{eqa1}-\eqref{eqa2}.
 
 In Sections \ref{S6} and \ref{S8} we study existence of an attractor and the dynamical properties of \eqref{eqa1}-\eqref{eqa2} (for parameters $\alpha,\beta$ large enough).  The solution in a neighborhood of this attractor develops a number of \emph{peaks} that eventually merge with each other while other peaks emerge, see Figure \ref{evolution}. We are able to bound from above the number of these peaks analytically.

\begin{figure}[t!]
    \centering
    \includegraphics[scale=0.36]{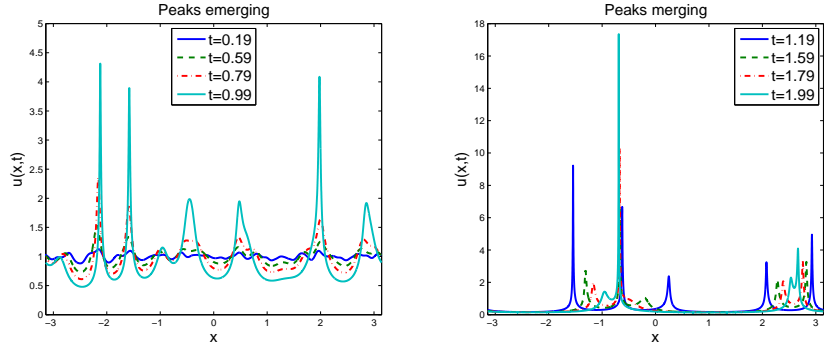}
\caption{Evolution in the case $\alpha=\beta=1$.}
\label{evolution}
\end{figure}

The aforementioned results are presented in Figure \ref{scheme}.

\begin{figure}[t!]
    \centering
    \includegraphics[scale=0.36]{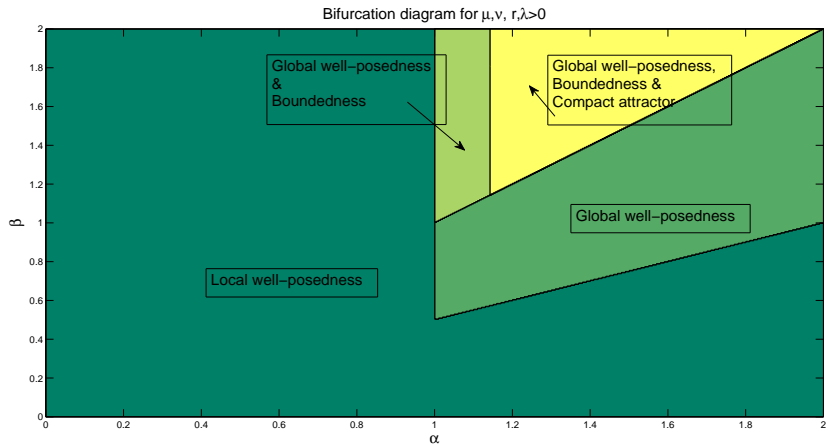}
\caption{Scheme assuming $\mu,\nu,r,\lambda>0$.}
\label{scheme}
\end{figure}

 Finally, in Section \ref{S7}, we perform a numerical study of \eqref{eqa1}-\eqref{eqa2} and provide numerical evidence of the finite time blow up in the case $\alpha=0.5$, $\beta=1$.

 \subsection{Novelties}
To the best of our knowledge, there are not many regularity results for the doubly parabolic fractional Keller-Segel system. Hence the generalization of the parabolic-elliptic global existence results of \cite{bournaveas2010one},  \cite{BG}, \cite{escudero2006fractional} to the system  \eqref{eqa1}-\eqref{eqa2}, even with $\beta = 2$ (the standard chemotactic term) and $r=0$, appears to be new. In particular, we prove global existence and boundedness of classical solutions with no restriction on size of initial data in the subcritical regime $\alpha >1$ and with a restriction in the critical case $\alpha =1$. The restriction of the latter result  is explicit and of the same order as the other parameters present in the system. In its proof we use the Wiener's algebra approach, which seems to be new in the Keller-Segel context. Nevertheless, we must admit that our smallness condition is quite stringent in the sense that it affects the entire Wiener's algebra norm (as opposed to merely the initial mass, for instance). 

The dynamical properties of the system are only known when $\alpha=\beta=2$, as far as we know. Moreover, the bound on the number of peaks seems new even in the classical $\alpha=\beta=2$ case.

\section{Some prior results}\label{S9}
Let us now present some literature concerning the Keller-Segel-type systems, in addition to that mentioned in subsection \ref{ssec:mot}.

\subsection{Keller-Segel system with classical diffusion}
There is a huge literature on the mathematical study of \eqref{eqa1local}-\eqref{eqa2local} and its parabolic-elliptic counterpart ($\tau=0$). Consequently, the list below is far from being exhaustive.
  
The global existence of solutions to \eqref{eqa1local}-\eqref{eqa2local} have been proved (under certain conditions) by many authors. In particular, Kozono \& Sugiyama \cite{KozonoSugiyama3} showed the global existence and decay of solutions to \eqref{eqa1local}-\eqref{eqa2local}, corresponding to small initial data in $d=3$ and with $1<r<1.5$ (see also \cite{KozonoSugiyama2}). Biler, Guerra \& Karch \cite{Bilerparabolicparabolic} recently proved that for every finite Radon measure there exist $\tau_0$ and a global in time mild solution for \eqref{eqa1local}-\eqref{eqa2local} with $\tau>\tau_0$. Corrias, Escobedo \& Matos \cite{CorriasEscobedo} proved that if the initial data $(u_0,v_0)$ is small in $L^1(\RR^2)\times \dot{H}^1(\RR^2)$ there exists a global solution. This result was recently generalized by Cao \cite{XCao}. Osaki \& Yagi \cite{OsakiYagi} and Osaki, Tsujikawa, Yagi \& Mimura \cite{Osaki2002119} obtained the existence of an exponential attractor while Hillen \& Potapov \cite{Hillen2}, using different techniques, also showed the global existence of solutions. 

Tello \& Winkler  \cite{TelloWinkler} proved the global existence of weak solutions for the parabolic-elliptic case with logistic term ($\tau=0,$ $r>0)$ for arbitrary $0\leq u_0\in L^\infty$; see also \cite{TelloWinkler2}. Winkler \cite{Winkler4} showed that there exists a global in time solution for the doubly parabolic case with a sufficiently strong logistic parameter $r$. He also obtained global weak solutions and studied the regularizing properties starting from merely $u_0\in L^1$ initial data in \cite{Winkler3}. Some finite time singularities results for solutions corresponding to certain initial data can be found in \cite{Winkler2}, \cite{Winkler} by Winkler.

\subsection{Spatio-temporal chaos}
A remarkable feature of the model \eqref{eqa1}-\eqref{eqa2} is its spatio-temporal chaotic behavior. In particular, the numerical solutions reported Painter \& Hillen \cite{Hillen4} for the system \eqref{eqa1local}-\eqref{eqa2local} develop a number of \emph{peaks} that emerge and, eventually, mix with other peaks. These peaks are maxima of $u,v$ that are very close to a region with their slope bigger than one. This phenomenon materializes in the numerical study of the system \eqref{eqa1}-\eqref{eqa2} with different values of $\alpha,\beta<2$ (see Figure \ref{evolution} for the case $\alpha=\beta=1$ and Section \ref{S7}.). As noted by Winkler in \cite{Winkler4}, the dynamical features of Keller-Segel models in high dimensions, in particular the existence of global attractors and bounded solutions, is an important topic. 

\subsection{Non-standard diffusions}\label{ssec:priorF}
The case of a nonlinear diffusion has been studied by several authors. See for instance Bedrossian \& Rodriguez \cite{BedNan}, Bedrossian, Rodriguez  \& Bertozzi \cite{Bertozzi}, Blanchet, Carrillo \& Lauren\c{c}ot \cite{blanchet2009critical} and Burczak, Cie\'slak \& Morales-Rodrigo \cite{BurczakCieslak}. 

The case of fractional powers of Laplacian instead of local derivatives in the first equation ($\alpha \in (0,2)$ and $\beta=2$) has been addressed by several authors. 

In particular, for the parabolic-elliptic case, Escudero \cite{escudero2006fractional} proved the boundedness of solutions in the one dimensional case with $\alpha>1$, while Li, Rodrigo \& Zhang \cite{li2010exploding} proved finite time singularities by constructing a particular set of initial data showing this behaviour. These authors also proved that any $L^1_tL^\infty_x$ bounded solution is global (see also \cite{AGM}). Bournaveas \& Calvez \cite{bournaveas2010one} studied the one-dimensional case with $0<\alpha<1$ and obtained the finite time blowup of solutions corresponding to big initial data and global solutions corresponding to small initial data. They also prove global existence for small data in the case $\alpha=1$, but here the problem of behavior of solutions emanating from large data remained an open question. Recently it was addressed in \cite{BurGraMEET} by the authors, where global-in-time smoothness without any smallness assumptions is proved.
In the context of the parabolic-elliptic case and similar problems, see also Ascasibar, Granero-Belinch\'on \& Moreno \cite{AGM}, Granero-Belinch\'on \& Orive \cite{GO}, and \cite{BG} by the authors. 

For the parabolic-elliptic system with fractional diffusion in the equation for $v$, compare Biler \& Karch \cite{BilKar10}.

The doubly parabolic case with fractional operators has been addressed by Biler \& Wu \cite{BilerWu} and Wu \& Zheng \cite{WuZheng}. In particular, these authors proved local existence of solutions, global existence of solutions for initial data satisfying some smallness requirements and ill-posedness in a variety of Besov spaces.

\section{Preliminaries}\label{ssub:prel}
Here we gather some basic terms used in what follows. We define
$$
\langle f \rangle=\frac{1}{|\TT|}\int_{\TT}f(x)dx, \quad \TT=[-\pi,\pi].
$$

\subsection{Singular integral operators and functional spaces}
We write $H$ for the Hilbert transform and $\Lambda =\sqrt{-\Delta}$, \emph{i.e.}
$$
\widehat{Hu}(k)=-i \,sgn(k) \, \hat{u}(k), \text{ and }\widehat{\Lambda^s u}(k)=|k|^s\hat{u}(k),
$$
where $\hat{\cdot}$ denotes the usual Fourier transform. Notice that $\Lambda=\pax H$  in one dimension and $\widehat{Hu}(0)=0$. The differential operator $\Lambda^{s}=(\sqrt{-\pax}^2)^{s}$ is defined by the action of the following kernels (see \cite{cor2} and the references therein):
\begin{equation} \label{lambda gamma}
\Lambda^{s} f(x)=c_{s}\;\text{p.v.} \int_{\mathbb{T}} \frac{f(x)-f(y)}{|x-y|^{1+s}}dy+c_{s}\sum_{k\in\mathbb{Z}\setminus\{0\}} \int_{\mathbb{T}} \frac{f(x)-f(y)}{|x-y+2k\pi|^{1+s}}dy,
\end{equation}
where $c_{s}>0$ is a normalization constant. In particular, in one dimension for $s=1$
\[
\Lambda f(x)=\frac{1}{2\pi} \; \text{p.v.} \int_{\mathbb{T}} \frac{f(x)-f(y)}{\sin^{2}\left((x-y)/2\right)}dy.
\]
\begin{remark}
Notice that given $v\in L^2(\TT)$, since $\langle H v \rangle=0$,
$$
\Lambda^{\beta-1} H v\in L^2(\TT)
$$
even if $\beta<1$.
\end{remark}

We write $H^s$ for the usual $L^2$-based Sobolev spaces with the norm
$$
\|f\|_{H^s}^2=\|f\|_{L^2}^2+\|f\|_{\dot{H}^s}^2, \quad \|f\|_{\dot{H}^s}=\|\Lambda^s f\|_{L^2}.
$$ 

The Wiener's algebra is defined as
\begin{equation}\label{11}
A(\TT)=\{\text{periodic functions }f\text{ such that }\hat{f}\in l^1\},
\end{equation}
\emph{i.e.}, the set of functions with absolutely convergent Fourier series. For a periodic function $u$, we define the Wiener's algebra-based seminorms:
$$
|u|_s=\sum_{k\in \ZZ}|k|^s|\hat{u}(k)|.
$$

\subsection{Sobolev embeddings and their constants}
Along the paper we are going to use different forms of Sobolev embedding (all of them classical). For the sake of clarity, we collect here these inequalities (and denote their constants) that are more often used. Assuming $\alpha>1$, we have for a function $g$ and a zero-mean value function $f$
\begin{equation}
\begin{aligned}
\|f\|_{L^{2/(\alpha-1)}}&\leq C^1_{SE}(\alpha)\|\Lambda^{1-\alpha/2}f\|_{L^2},\\
\|g\|_{L^{\infty}}&\leq C^2_{SE}(\alpha)\|g\|_{H^\frac{\alpha}{2}},\\
\|f\|_{L^{2+\frac{2\alpha-2}{2-\alpha}}}&\leq C^3_{SE}(\alpha)\|\Lambda^{(\alpha-1)/2}f\|_{L^2},\\
\|f\|_{W^{\frac{\alpha}{2} ,\infty}}&\leq C^4_{SE}(\alpha)\|f\|_{\dot{H}^{\alpha}}.
\end{aligned}
\label{C_SE2}
\end{equation}
\subsection{Notation}
We write $T_{max}$ for the maximum lifespan of the solution. 

For a given initial data $(u_0,v_0)$, we define
$$
\mathcal{N}=\max\{\|u_0\|_{L^1(\TT)},2\pi\}.
$$

\subsection{A notion of solution}

Let $u_0(x),v_0(x)\geq0$ be the initial data for the system \eqref{eqa1}-\eqref{eqa2}. Then we define its solution as follows

\begin{defi}\label{defipp}
Let $0<T<\infty$ be a positive parameter. The couple \[(u,v)\in L^\infty([0,T],L^2(\TT))\times L^\infty([0,T],H^{\beta/2}(\TT))\] is \emph{a solution of \eqref{eqa1}-\eqref{eqa2}} if
$$
\int_0^T\int_\TT [\pat\phi-\mu\Lambda^\alpha\phi]u+\pax\phi(u\Lambda^{\beta-1} H v) +\phi ru(1-u) \, dxdt - \int_\TT \phi(x,0) u_0 \, dx=0,
$$
$$
\int_0^T\int_\TT [\pat\varphi -\nu\Lambda^\beta\varphi-\lambda\varphi]v+\varphi u \,dxdt - \int_\TT \varphi(x,0) v_0 \, dx=0,
$$
for all test functions $\phi,\varphi \in C^\infty_c([0,T)\times\TT)$.
\end{defi}

\begin{defi}\label{defiglobal}
If a solution $(u,v)$ verifies the previous definition for every $0<T<\infty$, this solution is called \emph{a global solution}.
\end{defi}
Observe, that our notion of a global solution does not involve $T= \infty$.
In particular, our global solution may \emph{a priori} become arbitrarily large as time tends to infinity.

\section{Statement of results}\label{S1}

\subsection{Local-in-time existence, regularity and continuation criteria}
First, we have the following result
\begin{teo}\label{localexistence}
Let $s\geq3$, $\mu, \nu>0$ and $0< \beta,\alpha\leq 2$. If $(u_0,v_0)\in H^s(\TT)\times H^{s+\beta/2}(\TT)$ is the initial data for equation \eqref{eqa1}-\eqref{eqa2}, then it admits a unique solution 
$$
0\leq u\in C([0,T_{max}(u_0,v_0)],H^s(\TT))\cap L^2([0,T_{max}(u_0,v_0)],H^{s+\alpha/2}(\TT)),
$$
$$
0\leq v\in C([0,T_{max}(u_0,v_0)],H^{s+\beta/2}(\TT))\cap L^2([0,T_{max}(u_0,v_0)],H^{s+\beta}(\TT)).
$$
\end{teo}

Next, we prove the following continuation criteria, slightly stronger than the condition in Lemma 2.1 of \cite{Winkler}
\begin{teo}\label{continuation}
Assume that, for a finite time $T$ and initial data $(u_0,v_0)\in H^s(\TT)\times H^{s+\beta/2}(\TT),$ $s\geq3$, the solution to \eqref{eqa1}-\eqref{eqa2} satisfies
$$
\int_0^T\|\Lambda^\beta v(s)\|_{L^\infty}+\|\pax u (s)\|_{L^\infty} ds<\infty,
$$
then this solution can be continued up to time $T+\delta$ for a small enough $\delta>0.$ Moreover, if $\mu\geq\frac{1}{2\nu}$, $\alpha\geq\beta$, then, the previous condition may be replaced by
$$
\int_0^T\|u(s)\|^2_{L^\infty}+\|\Lambda^\beta v(s)\|_{L^\infty}+\|u(s)\|_{L^\infty}\|\Lambda^\beta v(s)\|_{L^\infty}ds<\infty.
$$
\end{teo}
Hence if $(u,v)$ is a solution showing finite time existence with  $T_{max}$  being its maximum lifespan, then we have
$$
\limsup_{t\rightarrow T_{max}} \|\Lambda^\beta v(t)\|_{L^\infty}+\|\pax u (t)\|_{L^\infty}=\infty.
$$
And, if $\mu\geq\frac{1}{2\nu}$, $\alpha\geq\beta$, the previous equation is replaced by
$$
\limsup_{t\rightarrow T_{max}} \|u(t)\|_{L^\infty}+\|\Lambda^\beta v(t)\|_{L^\infty}=\infty.
$$

Let us emphasize that the above results do not involve any \emph{extra} assumptions on the values of parameters $\alpha,\beta,r,\lambda$. They should be compared with Lemma 1.1 of \cite{Winkler4}.
\subsection{Global-in-time solutions}
\subsubsection{Small data regularity result for $\alpha \ge 1$, $r=0$}
Using the Wiener's algebra approach we obtain a global solution for small, periodic initial data. Recall that the Wiener's algebra is defined as in \eqref{11}.
\begin{teo}\label{globalwiener}Let $(u_0,v_0)\in H^3(\TT)\times H^{3+\beta/2}(\TT)$ and assume $1\leq\beta\leq2\leq1+\alpha$ as well as $r=0,\mu>1$ in the system \eqref{eqa1}-\eqref{eqa2}. Then, if the initial data satisfy
$$
|u_0|_{1}+|v_0 |_{\beta}<\min\{\mu-1,\nu-\langle u_0\rangle,\lambda/2\},
$$
the solution is global and 
$$
|u(t)|_{1}+|v(t) |_{\beta}\leq 
|u_0|_{1}+|v_0 |_{\beta}.
$$
\end{teo}
This result has the same flavor as \cite{Bae}, \cite{KozonoSugiyama3}. The case $\alpha=1$ is particularly interesting, because for the case $\alpha>1$, we prove below the existence of global solutions corresponding to arbitrary large initial data. Notice that the constant in the smallness condition depends explicitly on the parameters present in the problem and $\|u_0\|_{L^1}$. 

Theorem \ref{globalwiener} is stated for the case $r=0$. Let us recall, that in \cite{BG} we prove result for a similar, parabolic-elliptic system that involves an interplay between the admissible size of initial data and the logistic parameter $r$.

\subsubsection{Large data regularity result for $\alpha > 1$, $r \ge 0$}
We also have 
\begin{teo}\label{globalpp1}
Let $\mu,\nu>0$, $2\geq\alpha>1$, $r,\lambda\geq0,$ $2\geq\beta\geq\alpha/2$ and the initial data $(u_0,v_0)\in L^2(\TT)\times H^{\beta-\alpha/2}(\TT)$ be given. Then there exists at least one global in time weak solution corresponding to $(u_0,v_0)$ satisfying
$$
u\in L^\infty([0,T],L^2(\TT))\cap L^2([0,T],H^{\alpha/2}(\TT))\quad \forall \;T<\infty,
$$
$$
v\in L^\infty([0,T],H^{\beta-\alpha/2}(\TT))\cap L^2([0,T], H^{3\beta/2-\alpha/2}(\TT)) \quad \forall \;T<\infty.
$$
If, in addition, the initial data $(u_0,v_0)\in H^{k\alpha}(\TT)\times H^{k\alpha+\beta/2}(\TT)$, $k\in\NN$, $k\alpha\geq 3$, then there exists a unique global in time solution corresponding to $(u_0,v_0)$ that enjoys
$$
u\in C([0,T],H^{k\alpha}(\TT)) \quad \forall \;T<\infty,
$$
$$
v\in C([0,T],H^{k\alpha+\beta/2}(\TT)) \quad \forall \;T<\infty.
$$
\end{teo}

\subsubsection{Large data regularity result for $\alpha\geq1$, $\beta>0$, $r>0$.} The previous two results did not use additional regularity provided by the logistic term. It turns out that in presence of the logistic damping ($r>0$), for any $\beta>0$ one can obtain global solutions in both the critical case $\alpha=1$ (then weak solutions) and subcritical case $\alpha >1$ (then regular solutions). More precisely, we have
\begin{teo}\label{globalpp2}
Let $\mu,\nu, r>0$, $2\geq\alpha\geq1$, $\lambda\geq0,$ $2\geq\beta>0$ and the initial data $(u_0,v_0)\in L^2(\TT)\times H^{\beta/2}(\TT)$ be given. Then there exists at least one global in time weak solution corresponding to $(u_0,v_0)$ satisfying
$$
u\in L^\infty([0,T],L^2(\TT))\cap L^2([0,T],H^{\alpha/2}(\TT))\quad \forall \;T<\infty,
$$
$$
v\in L^\infty([0,T],H^{\beta/2}(\TT))\cap L^2([0,T], H^{\beta}(\TT)) \quad \forall \;T<\infty.
$$
If, in addition, $\alpha>1$, $\beta \le \alpha$ and the initial data $(u_0,v_0)\in H^{k\alpha}(\TT)\times H^{k\alpha+\beta/2}(\TT)$, $k\in\NN$, $k\alpha\geq 3$, then there exists a unique global in time solution corresponding to $(u_0,v_0)$ that enjoys
$$
u\in C([0,T],H^{k\alpha}(\TT)) \quad \forall \;T<\infty,
$$
$$
v\in C([0,T],H^{k\alpha+\beta/2}(\TT)) \quad \forall \;T<\infty.
$$
\end{teo}
\subsubsection{Absorbing set for $r>0$, $\beta \ge \alpha >1$ }

\begin{teo}\label{globalpp1abs}
Let $\mu,\nu>0$, $2\geq\beta \geq\alpha>1$, $\lambda\geq0$, $r>0$ and the initial data $(u_0,v_0)\in H^{k\alpha}(\TT)\times H^{k\alpha+\beta/2}(\TT)$, $k\in\NN$, $k\alpha\geq 3$ be given. Then there exist positive numbers $T^*$, $S(\cdot)$ such that
$$
\|u(t+1)\|^2_{\dot{H}^{k\alpha/2}}\leq S(\dot{H}^{k\alpha/2})\quad\forall\,t\geq T^*, \,0\leq k\alpha.
$$
\end{teo}

From Theorem \ref{globalpp1abs} follows in particular that
$$
\limsup_{t\rightarrow\infty}\|u(t)\|^2_{H^{\alpha/2}}\leq S(H^{\alpha/2})
$$
with $S(H^{\alpha/2})$ given by \eqref{SHalpha2}. Furthermore, it implies that  there exists $C$, depending on the parameters present in the problem and on the initial data, such that
$$
\max_{0\leq t<\infty}\{\|u(t)\|^2_{H^{k\alpha/2}}+\|v(t)\|^2_{H^{\beta+(k-1)\alpha/2}}\}\leq C,
$$
hence the solution is globally bounded.

\begin{remark}
We provide (and collect in \ref{AppB}) an estimate for numbers $S(\cdot)$ along the proof of Theorem \ref{globalpp1abs}. We provide them, since part of interest of this paper is to study the dynamical properties of the attractor of \eqref{eqa1}-\eqref{eqa2}. In particular, we believe that an estimate on the radius of the absorbing set and the number of peaks that may emerge is interesting (see below) and it requires a formula for $S(\cdot)$. Let us also notice that the hypothesis $\beta\geq\alpha$ is not required to obtain the existence of $S(L^2)$.
\end{remark}

\subsection{Instantaneous analyticity}
Before we state our result on the smoothing effect, we need some additional notation and definitions.
\subsubsection{Preliminaries}
Let us define
\begin{equation}\label{omega0}
0<\omega_0= \min\left\{\frac{\nu}{3},\frac{\mu}{8}\right\},
\end{equation}
and let $\omega$ be a positive constant that will be fixed later, depending on the parameters present in our problem and on the initial data. We define the (time dependent) complex strip 
$$
\Ss_\omega=\{x+i\xi,\;x\in\TT,\;|\xi|\leq \omega t\},
$$
and
\begin{equation}\label{tildeT}
\tilde{T}=\frac{1+\|u_0\|_{H^{3}(\TT)}^2+\|v_0\|_{H^{4}(\TT)}^2}{3\mathcal{K}},
\end{equation}
where $\mathcal{K}$ is given in \eqref{K}.

The Hardy-Sobolev space for the complex extension of a real function is given by
\begin{equation}\label{8}
H^s(\Ss_\omega)=\{f(z,t),\;z\in\Ss_\omega\; \text{s.t.}\;\|f\|_{H^s(\Ss_\omega)}<\infty \text{ and } f(x,t)\in\RR\},
\end{equation}
with norm

\begin{equation}\label{9}
\|f\|_{H^s(\Ss_\omega)}^2=\int_\TT |f(x\pm i\omega t)|^2dx+\int_\TT |\pax^s f(x\pm i\omega t)|^2dx.
\end{equation}

To motivate the use of these spaces, notice that if the $H^0(\Ss_\omega)=L^2(\Ss_\omega)$ norm of $u$ is bounded, then $u$ is analytic. In order to see this, we can use the Fourier series of $u$:
$$
u(x)=\sum_{k\in\ZZ} \hat{u}(k)e^{ikx}.
$$
Formally, evaluating it at $x=\alpha\pm i\omega t$, we have
$$
u(\alpha\pm i\omega t)=\sum_{k\in\ZZ} \hat{u}(k)e^{ik\alpha} e^{\mp k\omega t}.
$$
Then, since \[ \sum_{k\in\ZZ} |\hat{u}(k) |^2 e^{\mp k\omega t}  =  \int_\TT |u(\alpha\pm i\omega t)|^2 d\alpha < \infty, \] we have that
$\hat{u}(k)$ decays at least as $e^{\pm k\omega t}$. Since exponential decay for the Fourier series implies analyticity, we conclude the claim. For further details, see \cite{bakan2007hardy,BCG,ccfgl,CGO,GH,GNO}.

\subsubsection{Results}
We have

\begin{teo}\label{smoothingeffect}
Given $2\geq \alpha,\beta\geq 1$, $\mu,\nu>0$, $r\geq0$ and the initial data $(u_0,v_0)\in H^3 (\TT) \times H^4 (\TT)$, the solution $(u,v)$  of \eqref{eqa1}-\eqref{eqa2}  becomes real analytic for every $0<t< \tilde{T}$. Furthermore, the complex extension of  $(u,v)$ becomes complex analytic in the growing in time, complex strip $\Ss_\omega$ with $\omega\leq\omega_0$ and we have the bounds
$$
\|u(t)\|_{L^\infty(\Ss_\omega)}\leq \sqrt{2}\|u_0\|_{L^\infty(\TT)}, \;
\|v(t)\|_{L^\infty(\Ss_\omega)}\leq \sqrt{2}\|v_0\|_{L^\infty(\TT)}.
$$
\end{teo}
\begin{remark}If in addition $\min\{\alpha,\beta\}>1$, the restriction $\omega\leq\omega_0$ can be relaxed. 
\end{remark}
Theorem \ref{smoothingeffect} is local in time. However, using a classical bootstrapping argument, the analyticity of $(u,v)$ in a complex strip around the real axis (possible with a very small width of the strip) can be obtained for every positive time $t>0$.

More precisely, we have
\begin{corol}\label{coro1}
If $\alpha,\beta> 1$, and $\min\{\mu,\nu\}>0$, $r\geq0$, the solutions $(u,v)\in H^3(\TT)\times H^4(\TT)$ to the problem \eqref{eqa1}-\eqref{eqa2} are real analytic for every $0<t$.
\end{corol}
Moreover, it holds
\begin{corol}\label{illposed}
If $\alpha,\beta\geq 1$, and $\min\{\mu,\nu\}<0$, the problem is ill-posed, \emph{i.e.} there are solutions $(u,v)$ to the problem \eqref{eqa1}-\eqref{eqa2} such that 
$$
\|u_0\|_{H^3(\TT)}+\|v_0\|_{H^4(\TT)}<\epsilon
$$
and
$$
\limsup_{t\rightarrow\delta^{-}}\|u(t)\|_{H^3(\TT)}+\|v(t)\|_{H^4(\TT)}=\infty,
$$
for every $\epsilon>0$ and small $\delta>0$.
\end{corol}
\subsection{Large-time dynamics}

We are interested in the existence of attractors for \eqref{eqa1}-\eqref{eqa2} and their dynamical properties. 

\subsubsection{Existence of attractor}
We answer the question concerning the existence of a connected, compact attractor in the following theorem.

\begin{teo}\label{attractor}
Given $r>0$, $2\geq\beta\geq\alpha\geq 8/7$, the system \eqref{eqa1}-\eqref{eqa2} has a maximal, connected, compact attractor in the space $H^{3\alpha}(\TT)\times H^{3\alpha+\beta/2}(\TT)$.
\end{teo}

\subsubsection{Estimates of number of peaks}

We can apply Theorem \ref{smoothingeffect} to study certain dynamical properties of the system \eqref{eqa1}-\eqref{eqa2}. In particular, applying Theorem \ref{smoothingeffect} together with the complex analyticity of the solution $(u,v)$, we can obtain a bound on the number of peaks. Let us begin with

\begin{teo}\label{oscillations}
Let $N\geq3$, $N\in\NN$, $2\geq \alpha,\beta\geq 1$, $\mu,\nu>0$, $\lambda,r\geq0$ and the initial data $(u_0,v_0)\in H^3(\TT)\times H^4(\TT)$ be given and write
$$
\mathcal{W}=\frac{\omega_0 \tilde{T}}{N},
$$
where $\omega_0$ and $\tilde{T}$ are defined in \eqref{omega0} and \eqref{tildeT} respectively. Then, for any $\epsilon>0$, $0<\tilde{T}/(N-1)< t<\tilde{T}$, \[\TT=I^u_\epsilon\cup R^u_\epsilon=I^v_\epsilon\cup R^v_\epsilon,\] where $I^u_\epsilon,I^v_\epsilon$ are the union of at most $[\frac{4\pi}{\mathcal{W}}]$ intervals open in $\TT$, and
\begin{itemize}
\item $|\pax u(x)| \leq \epsilon, \text{ for all }x\in I^u_\epsilon,$
\item $\#\{x \in R^u_\epsilon : \pax u(x)=0\}\leq
\frac{2}{\log 2}\frac{2\pi}{\mathcal{W}}\log\left(\frac{\sqrt{2}(N-1)\|u_0\|_{L^\infty(\TT)}}{\mathcal{W}\epsilon}\right),$
\item $|\pax v(x)| \leq \epsilon, \text{ for all }x\in I^u_\epsilon,$
\item $\#\{x \in R^v_\epsilon : \pax v(x)=0\}\leq
\frac{2}{\log 2}\frac{2\pi}{\mathcal{W}}\log\left(\frac{\sqrt{2}(N-1)\|v_0\|_{L^\infty(\TT)}}{\mathcal{W}\epsilon}\right).$
\end{itemize}
\end{teo}
Notice that Theorem \ref{oscillations} gives us an estimate of the number of peaks appearing in the evolution (and reported in the numerical simulations). Indeed, we have the following corollary.

\begin{corol}\label{oscillationsattractor}
Let $r>0$, $2\geq\beta\geq\alpha\geq 8/7$ and $(u,v)$ be a solution in the attractor, then the number of peaks for $u$ can be bounded as
$$
\#\{\text{peaks for $u$}\}\leq
\frac{12\pi\mathcal{K}_1}{\log 2}\log\left(6\sqrt{2}\mathcal{K}_1C^2_{SE}(\alpha)S(H^{\alpha/2})\right),
$$
where $S(H^{\alpha/2})$ and $\mathcal{K}_1$ are defined in \eqref{SHalpha2} and \eqref{K1}, respectively.
\end{corol}
In \cite{Hillen4}, the authors perform a numerical study of the case $\alpha=\beta=2$, $\mu=\nu$, $r=\lambda$ and different values of $\lambda$ and $\nu$. They take initial data satisfying
$$
\|u_0\|_{L^\infty}=1,\;\|v_0\|_{L^\infty}\leq 1.01.
$$ 
We can use our previous results to give an analytical bound on the number of peaks that the solutions in \cite{Hillen4} develop:

\begin{corol}\label{oscillationsattractor2}
Let $(u,v)$ be the solution corresponding to the initial data in \cite{Hillen4}, then the number of peaks for $u$ and $v$ can be bounded as follows
$$
\#\{\text{peaks for $u$}\}\leq
\frac{12\pi\mathcal{K}_1}{\log 2}\log\left(6\sqrt{2}\mathcal{K}_1\right),
$$
$$
\#\{\text{peaks for $v$}\}\leq
\frac{12\pi\mathcal{K}_1}{\log 2}\log\left(6\sqrt{2}\mathcal{K}_11.01\right),
$$
where $\mathcal{K}_1=\mathcal{K}_1(2,2,\nu,\nu,\lambda,\lambda)$ is defined in \eqref{K1}. 
\end{corol}
Let us finally emphasize that presence of peaks is in no connection with a possibility of a blow up of solution. Our bounds are provided in the subcritical case, where global, regular solutions exist.

\subsection{Numerical simulations}
In the numerical part (Section \ref{S7}), we present first our simulations of emerging and merging peaks.
The main conclusion from our numerical study for further analysis is that, even in presence of a damping logistic term, our system may develop finite time singularities  for certain parameters. In particular, our numerics suggest that for $\alpha=0.5$, $\beta=1$ and a sufficiently strong nonlinear term (measured by chemical sensitivity $\chi$ there, compare the system \eqref{eqa1num} - \eqref{eqa2num}), the solution blows up in a finite time. Furthermore, our numerical solutions agree with the continuation criterion in Theorem \ref{continuation} in the sense that the spatial norms of the numerical solutions $\|\pax u\|_{L^\infty}$ are not integrable. 

\section{Proof of Theorem \ref{localexistence}: Local existence}\label{S2}
We prove now our local well-posedness result.

We prove the case $s=3$, the other cases being similar. \\
\textbf{Part 1. ({\emph{a priori} estimates}) }
We have
$$
\frac{1}{2}\frac{d}{dt}\|v\|_{L^2}^2\leq -\nu\|v\|_{\dot{H}^{\beta/2}}^2-\lambda\|v\|_{L^2}^2+\|u\|_{L^2}\|v\|_{L^2},
$$
\begin{align*}
\frac{1}{2}\frac{d}{dt}\|\Lambda^{\beta/2}\pax^3 v\|_{L^2}^2&=\int_{\TT}\Lambda^\beta \pax^3 v\pax^3\pat v\\
&= -\nu\|v\|_{\dot{H}^{3+\beta}}^2-\lambda\|v\|_{\dot{H}^{3+\beta/2}}^2+\|u\|_{\dot{H}^{3}}\|v\|_{\dot{H}^{3+\beta}}\\
&\leq -\frac{\nu}{2}\|v\|_{\dot{H}^{3+\beta}}^2-\lambda\|v\|_{\dot{H}^{3+\beta/2}}^2+\frac{2}{\nu}\|u\|^2_{\dot{H}^{3}},
\end{align*}
\begin{eqnarray*}
\frac{1}{2}\frac{d}{dt}\|u\|_{L^2}^2&=&-\mu\|u\|_{\dot{H}^{\alpha/2}}^2-\int_\TT\pax u u\Lambda^{\beta-1}Hvdx+r\int_\TT u^2(1-u)dx\\
&=&-\mu\|u\|_{\dot{H}^{\alpha/2}}^2+\frac{1}{2}\int_\TT u^2\Lambda^{\beta}vdx+r\int_\TT u^2(1-u)dx\\
&\leq& \|u(t)\|_{L^2}^2\|\Lambda^\beta v\|_{L^\infty}+r\|u\|_{L^2}^2,
\end{eqnarray*}
and
\begin{align*}
\frac{1}{2}\frac{d}{dt}\|\pax^3 u\|_{L^2}^2&=-\mu\|u\|_{\dot{H}^{3+\alpha/2}}^2+\int_\TT\pax^3 u \pax^4(u\Lambda^{\beta-1}Hv)dx\\
&\quad +r\int_\TT\pax^3 u \pax^3(u(1-u))dx,
\end{align*}
where the higher order terms are 
\begin{align*}
J_1&=\int_\TT\pax^3 u \pax^4u\Lambda^{\beta-1}Hvdx\leq \|\Lambda^\beta v\|_{L^\infty}\|u\|_{\dot{H}^{3}}^2,\\
J_2&=\int_\TT\pax^3 u u\Lambda^{\beta}\pax^3vdx\leq \|\Lambda^\beta \pax^3v\|_{L^2}\|u\|_{\dot{H}^{3}}\|u\|_{L^\infty},\\
J_3&=r\int_\TT (\pax^3 u)^2(1-2u) dx\leq r\|\pax^3u(t)\|_{L^2}^2.
\end{align*}
Using
$$
\int_\TT (\pax f)^4dx=-3\int\pax^2 f (\pax f)^2 f,
$$
together with H\"{o}lder's inequality, we obtain the following Gagliardo-Niremberg inequality
\begin{equation}\label{GN4}
\|\pax f\|_{L^4}^2\leq 3\|f\|_{L^\infty}\|\pax^2 f\|_{L^2}.
\end{equation}
Due to \eqref{GN4}, the lower order terms can be bounded as
$$
J_4=\int_\TT\pax^3 u \pax^3u\Lambda^{\beta}vdx\leq \|\Lambda^\beta v\|_{L^\infty}\|u\|_{\dot{H}^{3}}^2,
$$
\begin{align*}
J_5&=\int_\TT\pax^3 u \pax^2u\Lambda^{\beta}\pax vdx\leq \|\Lambda^\beta \pax v\|_{L^4}\|u\|_{\dot{H}^{3}}\|\pax^2 u\|_{L^4}\\
&\leq C\|\Lambda^\beta v\|^{0.5}_{L^\infty}\|\Lambda^\beta \pax^2 v\|^{0.5}_{L^2}\|u\|_{\dot{H}^{3}}^{1.5}\|\pax u\|_{L^\infty}^{0.5},
\end{align*}
$$
J_6=\int_\TT\pax^3 u \pax u\Lambda^{\beta}\pax^2 vdx\leq \|\Lambda^\beta \pax^2 v\|_{L^2}\|u\|_{\dot{H}^{3}}\|\pax u\|_{L^\infty},
$$
and
$$
J_7=6r\int_\TT\pax^3 u \pax u\pax^2 u dx\leq 6r\|u\|_{\dot{H}^{3}}\|\pax u\|_{L^4}\|\pax^2 u\|_{L^4}\leq C\|u\|_{H^{3}}^{2}\|u\|_{L^\infty}^{0.5}\|\pax u\|^{0.5}_{L^\infty}.
$$
We define the energy 
$$
E=\|u\|^2_{H^3}+\|v\|^2_{H^{3+\beta/2}}.
$$
We have
\begin{align}
\frac{d}{dt}E&\leq C (E+1)^3+\|\Lambda^\beta \pax^3v\|_{L^2}\|u\|_{\dot{H}^{3}}\|u\|_{L^\infty}\nonumber\\
&\quad -\frac{\nu}{2}\|v\|_{\dot{H}^{3+\beta}}^2-\lambda\|v\|_{\dot{H}^{3+\beta/2}}^2+\frac{2}{\nu}\|u\|^2_{\dot{H}^{3}}-\mu\|u\|_{\dot{H}^{3+\alpha/2}}^2
\nonumber\\
&\leq C(\nu) (E+1)^4 -\mu\|u\|_{\dot{H}^{3+\alpha/2}}^2 -\frac{\nu}{4}\|v\|_{\dot{H}^{3+\beta}}^2.\label{5}
\end{align}
Due to the previous inequality \eqref{5}, we obtain the desired bound for the energy. Moreover, from \eqref{5}, we get that $(u,v)\in L^2_t H^{3+\alpha/2}_x\times L^2_t H^{3+\beta}_x$. 

\textbf{Part 2. ({existence}) }
Once we have the energy estimates, we consider a family of Friedrichs mollifiers $\mathcal{J}_\epsilon$ and define the regularized initial data
$$
u^\epsilon(x,0)=\jeps u_0(x)\geq0,\,
v^\epsilon(x,0)=\jeps v_0(x)\geq0,
$$
and the regularized problems
\begin{eqnarray*}
\pat u^\epsilon & = & -\mu \jeps \Lambda^\alpha \jeps u^\epsilon+\jeps\pax\cdot(\jeps u^\epsilon(\Lambda^{\beta-1} H \jeps v^\epsilon)) +ru^\epsilon(1-u^\epsilon)\\ 
\pat v^\epsilon & = & -\jeps \nu\Lambda^\beta \jeps v^\epsilon-\lambda \jeps v^\epsilon+\jeps u^\epsilon. 
\end{eqnarray*}

Applying Picard's Theorem in $H^3\times H^{3+\beta/2}$, we find a solution $(u^\epsilon,v^\epsilon)$ to these approximate problems. These solutions exists up to time $T^\epsilon$. Furthermore, as $(u^\epsilon,v^\epsilon)$ verify the same energy estimates, we can take $T=T(u_0,v_0)$ independent of the regularizing parameter $\epsilon$. This concludes the existence part. \\
\textbf{Part 3. ({uniqueness}) }
To prove the uniqueness we argue by contradiction. Let us assume that there are two different solutions corresponding to the same initial data $(u_0,v_0)\in L^2\times H^{\beta/2}$. We write $(u_i,v_i)$, $i=1,2$ for these solutions and define $\bar{u}=u_1-u_2,$ $\bar{v}=v_1-v_2$. Then we have the bounds
$$
\frac{d}{dt}\|\bar{v}(t)\|_{H^{\beta/2}}^2+\nu\|\bar{v}(t)\|_{\dot{H}^{\beta}}^2\leq c(\nu)\|\bar{u}(t)\|_{L^2}^2,
$$
\begin{align*}
\frac{d}{dt}\|\bar{u}(t)\|_{L^2}^2&\leq 2\left|\int_\TT \pax \bar{u}[\bar{u}\Lambda^{\beta-1}Hv_1+u_2\Lambda^{\beta-1}H\bar{v}]dx\right|\\
&\leq \|\bar{u}(t)\|_{L^2}^2\|\Lambda^\beta v_1(t)\|_{L^\infty}+\|\bar{u}(t)\|_{L^2}\|\pax u_2(t)\|_{L^\infty}\|\Lambda^{\beta-1} \bar{v}(t)\|_{L^2}\\
&\quad +\|\bar{u}(t)\|_{L^2}\|u_2(t)\|_{L^\infty}\|\Lambda^{\beta} \bar{v}(t)\|_{L^2}.
\end{align*}
We use $\beta-1\leq \beta/2$ and Young's inequality to get
\begin{align*}
\frac{d}{dt}(\|\bar{u}(t)\|_{L^2}^2+\|\bar{v}(t)\|_{H^{\beta/2}}^2)&\leq (\|\bar{u}(t)\|_{L^2}^2+\|\bar{v}(t)\|_{H^{\beta/2}}^2)\\
&\quad\times\left(c(\nu)+\|\Lambda^\beta v_1(t)\|_{L^\infty}\right.\\
&\quad+\left.\frac{1}{2}\|\pax u_2(t)\|_{L^\infty}+c(\nu)\|u_2(t)\|_{L^\infty}^2\right).
\end{align*}
Finally we get 
\begin{align*}
\|\bar{u}(t)\|_{L^2}^2+\|\bar{v}(t)\|_{H^{\beta/2}}^2&\leq (\|\bar{u}_0\|_{L^2}^2+\|\bar{v}_0\|_{H^{\beta/2}}^2)\\
&\quad \times e^{[c(\nu)t+\int_0^t\|\Lambda^\beta v_1(s)\|_{L^\infty}+0.5\|\pax u_2(s)\|_{L^\infty}+c(\nu)\|u_2(s)\|_{L^\infty}^2]ds}.
\end{align*}
From the last inequality we obtain the uniqueness. \\
\textbf{Part 4. ({preservation of sign}) }
To finish the proof of the entire Theorem \ref{localexistence}, we need to show that for non-negative initial data the solution remains non-negative as well. 

To obtain pointwise bounds we apply the techniques developed in \cite{BG,cor2,CGO,GO,GH} and the references therein. Let $u(x,t)$ be a classical solution with a non-negative initial data and write $x_{t}\in \TT$ for a point such that $\min_x u(x,t)=u(x_t,t)$. Evaluating the equation \eqref{eqa1} at the point of minimum and using the kernel expression for $\Lambda^\alpha$, we have
$$
\frac{d}{dt}u(x_{t},t)\geq u(x_{t},t)\left[\Lambda^{\beta} v(x_{t},t)+r(1-u(x_{t},t))\right],\;t\geq 0,
$$
hence
$$
u(x_{t},t) \ge u_0(x_{0})e^{\int_0^t\Lambda^{\beta} v(x_{y},y)+r(1-u(x_{y},y)) dy}.
$$

Thus, $ u(x_{t},t) \ge 0$ since $u_0(x)\geq0$ and we conclude the claim. For the equation \eqref{eqa2} we can proceed similarly and we get
$$
v(x_t,t)\geq v_0(x)e^{-\lambda t}+e^{-\lambda t}\int_0^t u(x_s,s)ds\geq0.
$$
This ends the proof of Theorem  \ref{localexistence}.

Let us remark here that in the above proof and in the remainder of this text, we reinterpret certain terms in language of duality pairs, where there is no sufficient regularity to perform some intermediate computations. This applies in particular to the time derivative of a single (spatial) Fourier mode.
\section{Proof of Theorem \ref{continuation}: Continuation criterion}\label{S2b}

\textbf{Part 1.}
Let us write
\begin{equation}\label{6}
\int_0^T\|\Lambda^\beta v(s)\|_{L^\infty}+\|\pax u (s)\|_{L^\infty} ds=M<\infty.
\end{equation}
Now, assuming the finiteness of $M$, we need to obtain a bound for the energy
$$
E(t)=\|u(t)\|^2_{H^3}+\|v(t)\|^2_{H^{3+\beta/2}}.
$$
First notice that
\begin{align*}
\frac{1}{2}\frac{d}{dt}\|u\|_{L^2}^2 & \leq \frac{1}{2}\|\Lambda^\beta v\|_{L^\infty}\|u\|^2_{L^2}+r\|u\|_{L^2}^2,\\
\frac{1}{2}\frac{d}{dt}\|v\|_{L^2}^2 & \leq \frac{1}{2}\|u\|_{L^2}^2.\\
\end{align*}
Thus
\begin{align*}
\sup_{0\leq t \leq T}\|u(t)\|_{L^2}^2 & \leq \|u_0\|_{L^2}^2e^{M+rT},\\
\sup_{0\leq t \leq T}\|v(t)\|_{L^2}^2 & \leq (\|u_0\|_{L^2}^2e^M+\|v_0\|_{L^2}^2)e^T.
\end{align*}
Let $x_{ut}$ denote the point where $u(t)$ reaches its maximum and let $x_{vt}$ denote the point where $v(t)$ reaches its maximum. Then, $u(x_{ut},t)$ and $v(x_{vt},t)$ are Lipschitz functions and, as a consequence, applying Rademacher Theorem, are almost everywhere differentiable. Moreover, using the expressions for the kernels $\Lambda^s$ together with the positivity of $u$ and $v$, we have
\begin{align*}
\frac{d}{dt}u(x_{ut},t)& \leq u(x_{ut},t)\Lambda^\beta v(x_{ut},t)+ru(x_{ut},t),\\
\frac{d}{dt}v(x_{vt},t)& \leq u(x_{vt},t).
\end{align*}
As a consequence,
\begin{align*}
\sup_{0\leq t \leq T} \|u(t)\|_{L^\infty} & \leq \|u_0\|_{L^\infty}e^{M+rT},\\ 
\sup_{0\leq t \leq T}\|v(t)\|_{L^\infty} & \leq \|v_0\|_{L^\infty}+ \|u_0\|_{L^\infty}e^MT.
\end{align*}
Notice that to bound the lower order norms we have used merely 
$$
\int_0^T\|\Lambda^\beta v(s)\|_{L^\infty}ds<\infty.
$$ 
For the higher seminorm,
$$
y(t)=\|u(t)\|^2_{\dot{H}^3}+\|v(t)\|^2_{\dot{H}^{3+\beta/2}},
$$
due to energy estimates, we have
$$
\frac{dy}{dt}\leq  c(M,\nu)(\|\Lambda^\beta v\|_{L^\infty}+\|\pax u\|_{L^\infty}+\|u_0\|_{L^\infty}e^{M+rT})y(t),
$$
and using Gronwall's inequality we conclude the result in the case when $M$ of \eqref{6} is finite.

\textbf{Part 2.} To simplify notation we write
$$
\int_0^T\|u(s)\|^2_{L^\infty}+\|\Lambda^\beta v(s)\|_{L^\infty}+\|u(s)\|_{L^\infty}\|\Lambda^\beta v(s)\|_{L^\infty}ds=\tilde{M}.
$$
The idea for this second part is similar. Assuming the boundedness of $\tilde{M}$, it suffices to obtain global bounds for $M$ defined in \eqref{6}. To this end, we are going to use $\tilde{M}<\infty$ to bound $\|u(t)\|_{H^2}.$ Then we can use Sobolev's embedding to obtain the estimate for $M$. First we compute
$$
\frac{1}{2}\frac{d}{dt}\|\Lambda^{\beta}\pax^2v\|_{L^2}^2+\lambda\|\Lambda^{\beta}\pax^2v\|_{L^2}^2\leq-\nu\|\Lambda^{1.5\beta}\pax^2 v\|_{L^2}^2+\|\Lambda^{\beta/2}\pax^2u\|_{L^2}\|\Lambda^{1.5\beta}\pax^2v\|_{L^2},
$$
thus
$$
\frac{1}{2}\frac{d}{dt}\|\Lambda^{\beta}\pax^2v\|_{L^2}^2+\lambda\|\Lambda^{\beta}\pax^2v\|_{L^2}^2\leq\frac{1}{2\nu}\|\Lambda^{\beta/2}\pax^2u\|_{L^2}^2.
$$
Now we have
\begin{align*}
\frac{1}{2}\frac{d}{dt}\|\pax^2u\|_{L^2}^2 & \leq -\mu\|\Lambda^{\alpha/2}\pax^2 u\|_{L^2}^2+c\|\Lambda^\beta v\|_{L^\infty}\|\pax^2u\|_{L^2}^2\\
& \quad+c\|\pax^2u\|_{L^2}\left(\|u\|_{L^\infty}^{0.5}\|\pax^2u\|_{L^2}^{0.5}\|\Lambda^\beta v\|_{L^\infty}^{0.5}\|\Lambda^{\beta}\pax^2 v\|_{L^2}^{0.5}\right)\\
&\quad+\|u\|_{L^\infty}\|\Lambda^{\beta}\pax^2 v\|_{L^2}\|\pax^2u\|_{L^2}+r\|\pax^2u\|_{L^2}^2+2r\|\pax^2u\|_{L^2}\|\pax u\|_{L^4}^2.
\end{align*}
Using the previous bound and Young's inequality, we get
\begin{eqnarray*}
\frac{1}{2}\frac{d}{dt}\left(\|\pax^2u\|_{L^2}^2+\|\Lambda^{\beta}\pax^2v\|_{L^2}^2\right)&\leq& -\mu\|\Lambda^{\alpha/2}\pax^2 u\|_{L^2}^2+\frac{1}{2\nu}\|\Lambda^{\beta/2}\pax^2u\|_{L^2}^2\\
&&+c\|\Lambda^\beta v\|_{L^\infty}\|\pax^2u\|_{L^2}^2+\frac{1}{2\lambda}\|u\|_{L^\infty}^2\|\pax^2u\|_{L^2}^2\\
&&+c\|\pax^2u\|_{L^2}^2\|u\|_{L^\infty}^{0.5}\|\Lambda^\beta v\|_{L^\infty}^{0.5}\\
&&+c(\lambda)\|\pax^2u\|_{L^2}^2\|u\|_{L^\infty}\|\Lambda^\beta v\|_{L^\infty}\\
&&+cr\|\pax^2u\|_{L^2}^2\left(\|u\|_{L^\infty}+1\right).
\end{eqnarray*}
Hence we obtain a bound for the $H^2$ seminorm. In the same way we get a bound for the $L^2$ norm. Since we have a bound for $H^2$, using Sobolev embedding, we arrive at $$
\int_0^T\|\pax u(s)\|_{L^\infty}ds\leq c\int_0^T\|u(s)\|_{H^2}ds
\leq cT\|u_0\|_{H^2}\exp\left(c(\lambda)\tilde{M}\right).
$$
Theorem  \ref{continuation} is showed.

\section{Proof of Theorem \ref{globalwiener}: Global existence using the Wiener's algebra}\label{S10}
Our aim here is to prove Theorem  \ref{globalwiener}. We start this section with two preliminary results concerning lower order norms. The behaviour is quite different depending on the value of $r$. If $r=0$, we have 
\begin{lem}\label{L1norm}
Let $(u_0,v_0)$ be two non-negative, smooth initial data for equation \eqref{eqa1}-\eqref{eqa2} with $r=0$. Then, the solutions $(u,v)$ are non-negative functions. Moreover, if the initial data $(u_0,v_0)$ are in $L^1(\TT)$, the solutions $(u,v)$ verify
\begin{itemize}
\item
$
\|u(t)\|_{L^1(\TT)}=\|u_0\|_{L^1(\TT)}\quad \forall \;0\leq t\leq T_{max}
$
\item
$
\|v(t)\|_{L^1(\TT)}=\frac{\|u_0\|_{L^1(\TT)}}{\lambda}+\left(\|v_0\|_{L^1(\TT)}-\frac{\|u_0\|_{L^1(\TT)}}{\lambda}\right)e^{-\lambda t}\quad \forall \;0\leq t\leq T_{max}.
$
\end{itemize}
\end{lem}
For the sake of brevity we do not write the proof. For $r>0$ the analogous result reads (see also \cite{Hillen1}).
\begin{lem}\label{L1normr}
Let $(u_0,v_0)$ be two non-negative, smooth initial data for equation \eqref{eqa1}-\eqref{eqa2} with $r>0$. Let us define
$$
\mathcal{N}=\max\{\|u_0\|_{L^1(\TT)},2\pi\}.
$$
Then the solutions $(u,v)$ verify
\begin{itemize}
\item
$
\|u(t)\|_{L^1(\TT)}\leq \mathcal{N},\quad \forall \; 0\leq t\leq T_{max}$\\
$
\int_0^t\|u(s)\|^2_{L^2(\TT)}ds\leq \mathcal{N}t+2\mathcal{N},\quad \forall \; 0\leq t\leq T_{max},
$
\item
$
\|v(t)\|_{L^1(\TT)}\leq \max\{\|v_0\|_{L^1},\mathcal{N}/\lambda\},\quad \forall \;0\leq t\leq T_{max}.
$
\end{itemize}
\end{lem}
\begin{proof} We take $r=1$ without losing generality. The ODE for $\|u(t)\|_{L^1}$ is
\begin{equation}\label{normL1}
\frac{d}{dt}\|u(t)\|_{L^1}=\|u(t)\|_{L^1}-\|u(t)\|_{L^2}^2.
\end{equation}
Recalling Jensen's inequality
$
\|u(t)\|_{L^1}^2\leq 2\pi\|u(t)\|_{L^2}^2,
$
we get
$$
\frac{d}{dt}\|u(t)\|_{L^1}\leq \|u(t)\|_{L^1}\left(1-\frac{1}{2\pi}\|u(t)\|_{L^1}\right).
$$
From this inequality we conclude the first part of the result. Given $t>0$, we integrate \eqref{normL1} between $0$ and $t$ and obtain
$$
\|u(t)\|_{L^1}-\|u_0\|_{L^1}=\int_0^t\|u(s)\|_{L^1}ds-\int_0^t\|u(s)\|_{L^2}^2ds,
$$
thus
$$
\int_{0}^{t}\|u(s)\|_{L^2(\TT)}^2ds\leq \|u_0\|_{L^1}-\|u(t)\|_{L^1}
+\sup_{0\leq s\leq t}\|u(s)\|_{L^1}t\leq \mathcal{N}t+2\mathcal{N},
$$
and we conclude the second part. The bound for the $L^1$ norm of $v$ is straightforward and we get
$$
\|v(t)\|_{L^1(\TT)}\leq\frac{\mathcal{N}}{\lambda}\\
+\left(\|v_0\|_{L^1(\TT)}-\frac{\mathcal{N}}{\lambda}\right)e^{-\lambda t},\quad \forall \; t\geq0,\,(\lambda>0),
$$
or
$$
\|v(t)\|_{L^1(\TT)}\leq\mathcal{N}t+\|v_0\|_{L^1(\TT)},\quad \forall \;t\geq0,\, (\lambda=0).
$$
\end{proof}

Now we turn to the proof of Theorem  \ref{globalwiener}. Recall that we assume there $1\leq\beta\leq2\leq1+\alpha$ and $\mu>1, r=0$.

\begin{proof}[Proof of Theorem \ref{globalwiener}]

We denote by $\hat{f}(k)$ the $k-$th Fourier mode of a function $f$. Then, as stated in Lemma \ref{L1norm}, we have $\hat{u}(0,t)=\langle u_0\rangle$. We will study the evolution of 
$$
\mathcal{E}(t)=|u(t)|_{1}+|v(t)|_{\beta}.
$$
Our goal is to obtain (under appropriate assumptions) the maximum principle
\begin{equation}\label{7}
\mathcal{E}(t)\leq \mathcal{E}(0).
\end{equation}
Having this together with Fourier series that imply
\begin{align*}
\pax u&=\sum_{j}ij\hat{u}(j)e^{ijx}\Rightarrow \|\pax u\|_{L^\infty}\leq|u|_{1},\\
\Lambda^\beta v&=\sum_{j}|j|^{\beta}\hat{v}(j)e^{ijx}\Rightarrow \|\Lambda^\beta v\|_{L^\infty}\leq|v|_{\beta},
\end{align*}
we arrive at
$$
\int_0^T \|\pax u(s)\|_{L^\infty}+\|\Lambda^\beta v(s)\|_{L^\infty}ds\leq \mathcal{E}(0)T.
$$
Using the continuation argument in Theorem \ref{continuation}, we conclude the proof.

It remains to obtain the maximum principle \eqref{7}. The system \eqref{eqa1}-\eqref{eqa2} reads
\begin{align*}
\frac{d}{dt} |\hat{u}(k)||k|  &=  -\mu|k|^{1+\alpha} |\hat{u}(k)|+\frac{|k|\bar{\hat{u}}(k)}{|\hat{u}(k)|}\sum_{j}j\hat{u}(j)\frac{k-j}{|k-j|}|k-j|^{\beta-1} \hat{v}(k-j)\\
 &\quad +\frac{|k|\bar{\hat{u}}(k)}{|\hat{u}(k)|}\sum_{j}\hat{u}(k-j)|j|^{\beta} \hat{v}(j),
\end{align*}
$$
\frac{d}{dt} |\hat{v}(k)||k|^{\beta}   =  -\nu|k|^{2\beta} |\hat{v}(k)|-\lambda |\hat{v}(k)||k|^{\beta}+\frac{\bar{\hat{v}}(k)}{|\hat{v}(k)|}|k|^{\beta}\hat{u}(k),
$$
so, using Fubini-Tonelli Theorem
$$
\frac{d}{dt}|u(t)|_{1}\leq -\mu|u|_{1+\alpha}+2|u|_{1}|v|_{\beta}+|u|_{2}|v|_{\beta-1}+|u|_{0}|v|_{\beta+1},
$$
and
\begin{align*}
\frac{d}{dt}\mathcal{E}& \leq -\mu|u|_{1+\alpha}+2|u|_{1}|v|_{\beta}+|u|_{2}|v|_{\beta-1}+|u|_{0}|v|_{\beta+1}-\nu|v|_{2\beta}-\lambda |v|_{\beta}+|u|_{\beta}.
\end{align*}
Using Young's inequality and the assumptions we get
\begin{eqnarray*}
\frac{d}{dt}\mathcal{E}&\leq& (|v|_{\beta}+1-\mu)|u|_{2}+
(|u|_{1}+\langle u_0\rangle-\nu)|v|_{2\beta}+(2|u|_{1}-\lambda )|v|_{\beta}\\
&\leq& (\mathcal{E}+1-\mu)|u|_{2}+
(\mathcal{E}+\langle u_0\rangle-\nu)|v|_{2\beta}+(2\mathcal{E}-\lambda )|v|_{\beta},
\end{eqnarray*}
thus, if
$$
\mathcal{E}(0)<\min\{\mu-1,\nu-\langle u_0\rangle,\lambda/2\},
$$ 
we obtain a decay (consequently, a global bound) for $\mathcal{E}(t)$. 
\end{proof}
\section{Proof of Theorems \ref{globalpp1} and \ref{globalpp2}: Global existence for $\alpha\geq1$}\label{S3}
Now we proceed with the proof of the global existence of solutions for large data.

\begin{proof}[Proof of Theorem \ref{globalpp1}] Recall that $T$ is an arbitrary fixed number such that $0<T<\infty$. We will consider times $0\leq t\leq T$. As $\alpha>1$, we can take $\frac{\alpha-1}{2}=\delta>0$ as a fixed parameter. 

Let us outline the proof: in the first three steps, we obtain \emph{a priori} estimates. {\emph {i.e.}} we assume there that we have a solution $(u,v)$ as smooth as required. In Step 4, we construct the solutions as the limit of approximate problems satisfying the same \emph{a priori} estimates as in Steps 1, 2 and 3. 

\textbf{Step 1. ({\emph{a priori} estimates I}) }
In this step we obtain estimates showing  
\begin{align*}
u&\in L^\infty(0,T;L^2(\TT))\cap L^2(0,T;H^{\alpha/2}(\TT))\\
v&\in L^\infty(0,T;H^{\beta-\alpha/2}(\TT))\cap L^2(0,T;H^{3\beta/2-\alpha/2}(\TT)).
\end{align*}
Let us compute the evolution of the $L^2$ norm of $u$ in the case $r>0$. For $r=0$ the proof is analogous. We get
\begin{eqnarray*}
\frac{1}{2}\frac{d}{dt}\|u(t)\|_{L^2}^2&\leq& -\mu\|u(t)\|_{\dot{H}^{\alpha/2}}^2+\frac{1}{2}
\|\Lambda^{\alpha/2}(u(t))^2\|_{L^2}\|\Lambda^{\beta-\alpha/2} v(t)\|_{L^2}\\
&&+r\|u(t)\|_{L^2}^2-r\|u(t)\|_{L^3}^3\\
&\leq& -\mu\|u(t)\|_{\dot{H}^{\alpha/2}}^2+r\|u(t)\|_{L^2}^2-r\|u(t)\|_{L^3}^3\\
&&+C_{KP}(\alpha)\|u\|_{L^\infty}\|u\|_{\dot{H}^{\alpha/2}}\|\Lambda^{\beta-\alpha/2} v\|_{L^2}\\
&\leq& -\mu\|u(t)\|_{\dot{H}^{\alpha/2}}^2+r\|u(t)\|_{L^2}^2-r\|u(t)\|_{L^3}^3\\
&&+C_{KP}(\alpha)\langle u(t)\rangle\|u(t)\|_{\dot{H}^{\alpha/2}}\|\Lambda^{\beta-\alpha/2} v(t)\|_{L^2}\\
&&+C_{KP}(\alpha)C_{GN}(\alpha)\|u(t)-\langle u(t)\rangle\|_{L^1}^{\delta/(1+\delta)}\|u(t)\|_{\dot{H}^{\alpha/2}}^{(2+\delta)/(1+\delta)}\\
&&\times\|\Lambda^{\beta-\alpha/2} v(t)\|_{L^2},
\end{eqnarray*}
where we have used Lemma \ref{lemaaux2} together with the following interpolation inequality
\begin{equation}\label{C_GN}
\left|\|u\|_{L^\infty}-\langle u\rangle\right|\leq \|u-\langle u\rangle\|_{L^\infty}\leq C_{GN}(\alpha)\|u-\langle u\rangle\|_{L^1}^{\delta/(1+\delta)}\|u\|_{\dot{H}^{\alpha/2}}^{1/(1+\delta)}.
\end{equation}
Using Young's inequality and Lemmas \ref{L1norm} and \ref{L1normr}, we obtain
\begin{eqnarray}\label{uL2}
\frac{d}{dt}\|u(t)\|_{L^2}^2&\leq& -\mu\|u(t)\|_{\dot{H}^{\alpha/2}}^2+r\|u(t)\|_{L^2}^2-r\|u(t)\|_{L^3}^3+\frac{\mu}{2}\|u(t)\|_{\dot{H}^{\alpha/2}}^{2}\nonumber\\
&&+\frac{(C_{KP}(\alpha)C_{GN}(\alpha))^{\frac{2+2\delta}{\delta}}(2\mathcal{N})^{\frac{2+2\delta}{1+\delta}}\|\Lambda^{\beta-\alpha/2} v(t)\|_{L^2}^{\frac{2+2\delta}{\delta}}}{\mu}\nonumber\\
&&+\frac{(C_{KP}(\alpha)\mathcal{N})^2}{\mu}\|\Lambda^{\beta-\alpha/2} v(t)\|_{L^2}^{2}.
\end{eqnarray}

Now, fix $t>0$ and consider the equation for the $k$-th Fourier coefficient of $v$
$$
\frac{d}{dt}\hat{v}(k,t)=-\nu|k|^\beta\hat{v}(k,t)-\lambda\hat{v}(k,t)+\hat{u}(k,t).
$$ 
Solving this ODE, we get
\begin{equation}\label{fourierv}
e^{(\nu|k|^\beta+\lambda)t}\hat{v}(k,t)=\hat{v_0}(k)+\int_0^t e^{(\nu|k|^\beta+\lambda)s}\hat{u}(k,s)ds.
\end{equation}

As $v_0\in H^{\gamma}$ with $\gamma=\beta-\alpha/2<\beta-0.5$, using \eqref{fourierv}, we get
\begin{align}
|k|^{\beta-\alpha/2} e^{(\nu|k|^\beta+\lambda)t}|\hat{v}(k,t)|&\leq |k|^{\beta-\alpha/2}|\hat{v_0}(k)|\nonumber\\
&\quad +\int_0^t |k|^{\beta-\alpha/2} e^{(\nu|k|^\beta+\lambda)s}|\hat{u}(k,s)|ds\nonumber\\
&\leq |k|^{\beta-\alpha/2}|\hat{v_0}(k)|+\frac{|k|^{\beta-\alpha/2}}{\nu|k|^\beta+\lambda}\mathcal{N}e^{(\nu|k|^\beta+\lambda)t},\label{fourierv2}
\end{align}
hence
$$
\int_0^t\|\Lambda^{\beta-\alpha/2} v(s)\|^p_{L^2}ds\leq t C(\alpha,\beta,\lambda,\|u_0\|_{L^1(\TT)},\|v_0\|_{H^{\beta-\alpha/2}(\TT)},\nu,p).
$$ 
Consequently, using Lemma \ref{L1normr}, we have that
\begin{multline*}
\|u(t)\|_{L^2}^2+\frac{\mu}{2}\int_0^t\|u(s)\|_{\dot{H}^{\alpha/2}}^2+r\|u(s)\|_{L^3}^3 ds\\
\leq \|u_0\|_{L^2}^2+\mathcal{N}t+2\mathcal{N}+t C(\alpha,\beta,\nu,\lambda,\|u_0\|_{L^1(\TT)},\|v_0\|_{H^{\beta-\alpha/2}(\TT)}).
\end{multline*}
In the case $\lambda>0$ we obtain simply
$$
\|v(t)\|^2_{L^2}+2\nu\int_0^t\|v(s)\|^2_{\dot{H}^{\beta/2}}ds\leq \|v_0\|^2_{L^2}+c(\lambda)\int_0^t\|u(s)\|_{L^2}^2ds.
$$

Testing the equation for $v$ against $\Lambda^{2\beta-\alpha} v$ and using self-adjointness we get
$$
\frac{1}{2}\frac{d}{dt}\|v(t)\|_{\dot{H}^{\beta-\alpha/2}}^2+\lambda\|v(t)\|_{\dot{H}^{\beta-\alpha/2}}^2+\nu\|v(t)\|_{\dot{H}^{3\beta/2-\alpha/2}}^2\leq \|u\|_{\dot{H}^{\alpha/2}}\|v(t)\|_{\dot{H}^{2\beta-\alpha-\alpha/2}}.
$$
As $\beta\leq 2<2\alpha$, we get $2\beta-\alpha-\alpha/2\leq 1.5\beta-\alpha/2$ and we can use Young's and Poincar\'e's inequalities to conclude this step.

\textbf{Step 2. ({\emph{a priori} estimates II}) } In this part we obtain estimates showing  
\begin{align*}
u&\in L^\infty(0,T;H^{\alpha/2}(\TT))\cap L^2(0,T;H^{\alpha}(\TT)),\\
v&\in L^\infty(0,T;H^{\beta/2+\alpha/2}(\TT))\cap L^2(0,T;H^{\beta+\alpha/2}(\TT)).
\end{align*}
Testing the equation for $v$ against $\Lambda^{\alpha+\beta}v$, we get
$$
\frac{1}{2}\frac{d}{dt}\|v(t)\|_{\dot{H}^{\beta/2+\alpha/2}}^2+\nu\|v(t)\|_{\dot{H}^{\beta+\alpha/2}}^2\leq \|u(t)\|_{\dot{H}^{\alpha/2}}\|v(t)\|_{\dot{H}^{\beta+\alpha/2}}.
$$
The previous inequality implies
\begin{multline*}
\|v(t)\|^2_{\dot{H}^{\beta/2+\alpha/2}}+\nu\int_0^t\|v(s)\|^2_{\dot{H}^{\beta+\alpha/2}}ds\\
\leq\|v_0\|^2_{\dot{H}^{\beta/2+\alpha/2}}+ 
c(\nu)\int_0^t\|\Lambda^{\alpha/2}u(s)\|_{L^2}^2\leq \|v_0\|^2_{\dot{H}^{\beta/2+\alpha/2}}+C
\end{multline*}
with constant 
$$
C=C(\alpha,\beta,\mu,\nu,\lambda,\|u_0\|_{L^1(\TT)},\|u_0\|_{L^2(\TT)},\|v_0\|_{H^{\beta-\alpha/2}(\TT)},T).
$$
We compute
\begin{eqnarray*}
\frac{d}{dt}\|u(t)\|_{\dot{H}^{\alpha/2}}^2&=&-\mu\int_\TT|\Lambda^{\alpha}u|^2dx+
\int_{\TT}u\Lambda^\beta v\Lambda^{\alpha}u dx+r\int_\TT|\Lambda^{\alpha/2}u|^2dx\\
&&+
\int_{\TT}\pax u\Lambda^{\beta-1}H v\Lambda^{\alpha}u dx-r\int_{\TT}\Lambda^{\alpha/2}(u^2)\Lambda^{\alpha/2}udx\\
&\leq& -\frac{3\mu}{4}\|u(t)\|_{\dot{H}^{\alpha}}^2+
c(\mu)\|u(t)\|_{L^2}^2\|\Lambda^\beta v(t)\|_{L^\infty}^2\\
&&+r\|u(t)\|_{\dot{H}^{\alpha/2}}^2+\|u\|_{\dot{H}^1}\|\Lambda^{\beta-1}H v\|_{L^\infty}\|u\|_{\dot{H}^\alpha}\\
&&+c(r)\|\Lambda^{\alpha/2}u(t)\|_{L^\infty}\|u\|_{L^2}\|u(t)\|_{\dot{H}^{\alpha/2}}.
\end{eqnarray*}
Using the interpolation inequalities 
\begin{align*}
\|u\|_{\dot{H}^1}&\leq c\|u\|_{\dot{H}^\alpha}^{1/\alpha}\|u\|_{L^2}^{(\alpha-1)/\alpha},\\
\|\Lambda^{\alpha/2}u\|_{L^\infty}&\leq c\|u\|_{\dot{H}^\alpha},\\
\|\Lambda^\beta v(t)\|_{L^\infty}^2&\leq c\|v(t)\|_{\dot{H}^{\beta+\alpha/2}}^2,\\
\|\Lambda^{\beta-1}H v\|_{L^\infty}&\leq c\|v\|_{\dot{H}^{\beta-1+\alpha/2}}\leq c\|v\|_{\dot{H}^{\beta/2+\alpha/2}},
\end{align*}
we obtain
$$
\|u(t)\|_{\dot{H}^{\alpha/2}}^2 +\frac{\mu}{2}\int_0^t\|u(t)\|_{\dot{H}^{\alpha}}^2 \leq \|u_0\|_{\dot{H}^{\alpha/2}}^2+C,
$$
where the constant depends on
$$
C=C(\alpha,\beta,\mu,\nu,\lambda,\|u_0\|_{L^1(\TT)},\|u_0\|_{L^2(\TT)},\|v_0\|_{H^{\beta-\alpha/2}(\TT)},\|v_0\|_{H^{\beta/2+\alpha/2}(\TT)},T).
$$
Notice that, in the case $\alpha>1.5$, we have
$$
\int_0^T\|\pax u(t)\|_{L^\infty}+\|\Lambda v(t)\|_{L^\infty}dt\leq c\int_0^T\|u(t)\|_{H^\alpha}+\|v(t)\|_{H^{\beta+\alpha/2}}dt\leq C,
$$
so, in this case, we are done with the entire proof.

\textbf{Step 3. ({\emph{a priori} estimates III}) }  In this step we obtain that
\begin{align*}
u&\in L^\infty(0,T;H^\alpha(\TT))\cap L^2(0,T;H^{3\alpha/2}(\TT)),\\
v&\in L^\infty(0,T;H^{\beta/2+\alpha}(\TT))\cap L^2(0,T;H^{\beta+\alpha}(\TT)).
\end{align*}
Testing the equation for $v$ against $\Lambda^{\beta+2\alpha} v$, we obtain
$$
\|v(t)\|^2_{\dot{H}^{\beta/2+\alpha}}+\nu\int_0^t\|v(s)\|^2_{\dot{H}^{\beta+\alpha}}ds\\
\leq\|v_0\|^2_{\dot{H}^{\beta/2+\alpha}}+ C.
$$
We compute
\begin{eqnarray*}
\frac{d}{dt}\|u(t)\|_{\dot{H}^{\alpha}}^2&=&-\mu\int_\TT|\Lambda^{3\alpha/2}u|^2dx+
\int_{\TT}u\Lambda^\beta v\Lambda^{2\alpha}u dx+r\int_\TT|\Lambda^{\alpha}u|^2dx\\
&&+
\int_{\TT}\pax u\Lambda^{\beta-1}H v\Lambda^{2\alpha}u dx-r\int_{\TT}\Lambda^{\alpha}(u^2)\Lambda^{\alpha}udx\\
&\leq&-\mu\|u(t)\|_{\dot{H}^{3\alpha/2}}^2+\|\Lambda^{\alpha/2}(u\Lambda^\beta v)\|_{L^2}\|u(t)\|_{\dot{H}^{3\alpha/2}}\\
&&+r\|u(t)\|_{\dot{H}^{\alpha}}^2+\|\Lambda^{\alpha/2}(\pax u\Lambda^{\beta-1}H v)\|_{L^2}\|u(t)\|_{\dot{H}^{3\alpha/2}}\\
&&+c\|u(t)\|_{L^\infty}\|u(t)\|_{\dot{H}^{\alpha}}^2\\
&\leq&-\frac{\mu}{2}\|u(t)\|_{\dot{H}^{3\alpha/2}}^2+c(\mu)[\|u(t)\|^2_{\dot{H}^{\alpha/2}}\|\Lambda^\beta v\|_{L^\infty}^2\\
&&+\|u(t)\|_{L^\infty}^2\|\Lambda^{\beta+\alpha/2} v\|_{L^2}^2]+r\|u(t)\|_{\dot{H}^{\alpha}}^2\\
&&+c(\mu)[\|u(t)\|_{\dot{H}^{1+\alpha/2}}^2\|\Lambda^{\beta-1}H v\|_{L^\infty}^2+\|\pax u\|_{L^2}^2\|\Lambda^{\beta-1+\alpha/2} v\|_{L^\infty}^2]\\
&&+c\|u(t)\|_{L^\infty}\|u(t)\|_{\dot{H}^{\alpha}}^2.
\end{eqnarray*}
\textbf{Step 4. (construction of a solution)} If the initial data $(u_0,v_0)\in H^{k\alpha}(\TT)\times H^{k\alpha+\beta/2}(\TT)$, $k\in\NN$, $k\alpha\geq 3$ we have local existence of regular solutions from Theorem \ref{localexistence}. Additionally, Step 3 gives us bounds 
\begin{align*}
u&\in L^\infty(0,T;H^\alpha(\TT))\cap L^2(0,T;H^{3\alpha/2}(\TT)),\\
v&\in L^\infty(0,T;H^{\beta/2+\alpha}(\TT))\cap L^2(0,T;H^{\beta+\alpha}(\TT))
\end{align*}
that are independent from the local time of existence; let us call them global-in-time bounds. In fact, to obtain the global-in-time bounds rigorously, using regularity given by Theorem  \ref{localexistence}, we need at step 3 to reinterpret some intermediate steps in terms of duality pairing. A similar remark applies to obtaining the evolutionary norms in all steps 1-3, including the notion of the time derivative of a single (spatial) Fourier mode. This point has been already raised by the end of Section \ref{S2}.  Our global-in-time bounds give
$$
\int_0^T\|\pax u(s)\|_{L^\infty}ds\leq C(T).
$$
To conclude with the continuation criterion given by Theorem \ref{continuation}, we need also 
\[\int_0^T \|\Lambda^\beta v (s)\|_{L^\infty} ds \le C(T).\] In fact, using $\beta-1+\alpha/2\leq\beta/2+\alpha/2$ we obtain
\begin{align*}
\|\Lambda^{\beta-1+\alpha/2} v(t)\|_{L^\infty}^2, \; \|\Lambda^{\beta-1}H v (t) \|_{L^\infty}^2\in\;& L^\infty, \\
\|\Lambda^{\beta+\alpha/2} v(t)\|_{L^2}^2, \;\|\Lambda^\beta v (t)\|_{L^\infty}^2\in\;& L^1
\end{align*}
Next, let us consider the case where the initial data is not that smooth, but merely $(u_0,v_0)\in L^2\times H^{\beta-\alpha/2}$. After mollification, we have an initial data $(u^\epsilon_0,v^\epsilon_0)$ with the desired regularity. Applying the previous reasoning, we have a global smooth regularized solution $(u^\epsilon,v^\epsilon)$. Due to Step 1, these functions are uniformly bounded in 
$$
u^\epsilon \in L^\infty([0,T],L^2(\TT))\cap L^2([0,T],H^{\alpha/2}(\TT)),
$$
$$
v^\epsilon \in L^\infty([0,T],H^{\beta-\alpha/2}(\TT))\cap L^2([0,T], H^{3\beta/2-\alpha/2}(\TT)).
$$
Testing $\pat u^\epsilon,\pat v^\epsilon$ against $\phi\in H^2$ and using the duality pairing, we obtain a uniform bound
$$
\pat u^\epsilon, \; \pat v^\epsilon\in L^\infty([0,T],H^{-2}(\TT)).
$$
Applying Aubin-Lions's Theorem (with $H^{\alpha/2}\subset \subset L^2\subset H^{-2}$ for $u^\epsilon$ and $H^{3\beta/2-\alpha/2}\subset \subset L^2 \subset H^{-2}$ for $v^\epsilon$), we take a subsequence (denoted again by $\epsilon$) such that
$$
u^\epsilon(t)\rightarrow u(t) \text{ in }L^2_tL^2_x, \quad
u^\epsilon(t)\rightharpoonup u(t) \text{ in }L^2_tH^{\alpha/2}_x,
$$
$$
v^\epsilon(t)\rightarrow v(t) \text{ in }L^2_tL^2_x , \quad
v^\epsilon(t)\rightharpoonup v(t) \text{ in }L^2_tH^{3\beta/2-\alpha/2}_x.
$$
Using the properties of the mollifier, we have
$$
u^\epsilon(0)\rightarrow u_0 \text{ in }L^2,\;\;v^\epsilon(0)\rightarrow v_0 \text{ in }L^2.
$$
With the previous strong convergence, we can pass to the limit in the weak formulations of Definition \ref{defipp}.
\end{proof}
We deal now with the existence of a global solution in the critical and subcritical cases $\alpha\geq1$, where the logistic term is arbitrarily weak but positive ($r>0$).
\begin{proof}[Proof of Theorem \ref{globalpp2}]  
We begin the proof with a new first \emph{a priori} estimate, that provides global weak solutions for  $\alpha\geq1$. Next we follow the proof of previous theorem to show existence of strong solutions in the case  $\alpha >1$. \\
\textbf{Step 1. ({\emph{a priori} estimates I} and weak solutions) } 

Let us consider times $0\leq t\leq T$ where $T$ is an arbitrary fixed number.  

Let $\alpha\geq1$, $\beta>0$. Testing the equation for $v$ with $\Lambda^\beta v$, one obtains
\begin{align*}
\frac{1}{2}\frac{d}{dt}\|v(t)\|_{\dot{H}^{\beta/2}}^2+\lambda \|v(t)\|_{\dot{H}^{\beta/2}}^2&\leq -\frac{\nu}{2}\int_\TT|\Lambda^{\beta}v|^2dx+ \frac{1}{2\nu}\|u\|_{L^2}^2.
\end{align*}
This inequality, together with Lemma \ref{L1normr} implies
$$
\int_0^t\|v\|^2_{H^{\beta}}ds\leq C(T,\mathcal{N},r,\nu).
$$
Testing now the equation for $u$ with $u$, we have
\begin{align*}
\frac{1}{2}\frac{d}{dt}\|u(t)\|_{L^2}^2&=-\mu\int_\TT|\Lambda^{\alpha/2}u|^2dx+\frac{1}{2}
\int_{\TT}u^2\Lambda^\beta vdx+r\|u(t)\|_{L^2}^2-r\|u(t)\|_{L^3}^3\\
&\leq-\mu\|u\|_{\dot{H}^{\alpha/2}}^2+\frac{1}{2}\left(\|u-\langle u\rangle\|_{L^4}^2+ \sqrt{2 \pi} \langle u \rangle \|u\|_{L^2}\right)
\|v\|_{\dot{H}^{\beta}}+r\|u(t)\|_{L^2}^2\\
&\leq-\mu\|u\|_{\dot{H}^{\alpha/2}}^2+ C\left(\|u-\langle u\rangle\|_{L^2}\|u\|_{\dot{H}^{0.5}}+ \langle u \rangle \|u\|_{L^2}\right)
\|v\|_{\dot{H}^{\beta}}\\
&\quad +r\|u(t)\|_{L^2}^2\\
&\leq-\mu\|u\|_{\dot{H}^{\alpha/2}}^2+ C\left(\|u\|_{L^2}\|u\|_{\dot{H}^{0.5}}+\langle u \rangle \|u\|_{L^2}\right)
\|v\|_{\dot{H}^{\beta}}+r\|u(t)\|_{L^2}^2,
\end{align*}
where we have used the inequality
$$
\|u-\langle u\rangle\|_{L^4}^2\leq C\|u-\langle u\rangle\|_{L^2}\|u\|_{\dot{H}^{0.5}}\leq C\|u\|_{L^2}\|u\|_{\dot{H}^{0.5}}.
$$
Young's inequality implies
\begin{align*}
\frac{d}{dt}\|u(t)\|_{L^2}^2&\leq-\frac{\mu}{2}\|u\|_{\dot{H}^{\alpha/2}}^2+C(r,\mathcal{N})\|u\|^2_{L^2}
(\|v\|_{\dot{H}^{\beta}}^2+1),
\end{align*}
thus
\begin{align}\label{ubound}
u&\in L^\infty(0,T;L^2(\TT))\cap L^2(0,T;H^{\alpha/2}(\TT))\\
v&\in L^\infty(0,T;H^{\beta/2}(\TT))\cap L^2(0,T;H^{\beta}(\TT))\label{vbound}.
\end{align}

After mollification, we have a smooth initial data $(u^\epsilon_0,v^\epsilon_0)$. Let us consider the approximate problems
\begin{eqnarray*}
\pat u^\epsilon & = & -\mu\Lambda^\alpha u^\epsilon+\pax(u^\epsilon\Lambda^{\beta-1} H v^\epsilon) +ru^\epsilon(1-u^\epsilon)+\epsilon \pax^2u^\epsilon,\\ 
\pat v^\epsilon & = & -\nu\Lambda^\beta v^\epsilon-\lambda v^\epsilon+u^\epsilon+\epsilon \pax^2v^\epsilon.
\end{eqnarray*}
We have local existence of regular solutions from Theorem \ref{localexistence}. Furthermore, due to Theorem \ref{globalpp1}, we have a global-in-time, smooth regularized solution $(u^\epsilon,v^\epsilon)$. Due to \eqref{ubound} and \eqref{vbound}, these functions are uniformly bounded in 
$$
u^\epsilon \in L^\infty([0,T],L^2(\TT))\cap L^2([0,T],H^{\alpha/2}(\TT)),
$$
$$
v^\epsilon \in L^\infty([0,T],H^{\beta/2}(\TT))\cap L^2([0,T], H^{\beta}(\TT)).
$$
As in the proof of Theorem \ref{globalpp1},
$$
\pat u^\epsilon, \; \pat v^\epsilon\in L^\infty([0,T],H^{-2}(\TT)).
$$
Applying Aubin-Lions's Theorem (with $H^{\alpha/2}\subset \subset L^2\subset H^{-2}$ for $u^\epsilon$ and $H^{\beta}\subset \subset L^2\subset H^{-2}$ for $v^\epsilon$), we take a subsequence (denoted again by $\epsilon$) such that
$$
u^\epsilon(t)\rightarrow u(t) \text{ in }L^2_tL^2_x, \quad
u^\epsilon(t)\rightharpoonup u(t) \text{ in }L^2_tH^{\alpha/2}_x,
$$
$$
v^\epsilon(t)\rightarrow v(t) \text{ in }L^2_tL^2_x, \quad
v^\epsilon(t)\rightharpoonup v(t) \text{ in }L^2_tH^{\beta}_x.
$$
Using the properties of the mollifier, we can pass to the limit in the weak formulations of Definition \ref{defipp}. \\
\textbf{Step 2. (Further {\emph{a priori} estimates and strong solutions}) } 
The weak regularity of the previous step implies for $\beta \le \alpha$ that $u,v$ enjoy weak regularity of Step 1 of Theorem \ref{globalpp1}. Therefore we can rewrite Steps 2. - 4. of Theorem \ref{globalpp1} and obtain $$
u\in C([0,T],H^{k\alpha}(\TT)) \quad \forall \;T<\infty,
$$
$$
v\in C([0,T],H^{k\alpha+\beta/2}(\TT)) \quad \forall \;T<\infty.
$$
Let us observe that we need $\alpha > 1$ for the interpolations and embeddings at the beginning of Step 2 of Theorem \ref{globalpp1}.
\end{proof}
\section{Proof of Theorem \ref{globalpp1abs}: absorbing set.}\label{pf:abs}
The proof uses the estimates from the proof of Theorem  \ref{globalpp1}. Let us  write
\begin{equation}\label{C_FS}
C_{FS}(\beta,\alpha,\lambda,\nu)=\sum_{k\in\ZZ}\left(\frac{|k|^{\beta-\alpha/2}}{\nu|k|^\beta+\lambda}\right)^2
\end{equation}
According to \eqref{fourierv2}, for every $t\geq0$, we have
$$
\|v(t)\|_{\dot{H}^{\beta-\alpha/2}}^2\leq \|v_0\|_{\dot{H}^{\beta-\alpha/2}}^2e^{-\lambda t}+\mathcal{N}C_{FS}(\beta,\alpha,\lambda,\nu),
$$
so
\begin{eqnarray*}
\int_t^{t+1}\|v(s)\|_{\dot{H}^{\beta-\alpha/2}}^2ds&\leq& \frac{\|v_0\|_{\dot{H}^{\beta-\alpha/2}}^2}{\lambda}(1-e^{-\lambda})e^{-\lambda t}+\mathcal{N}C_{FS}(\beta,\alpha,\lambda,\nu)\\
&\leq& \frac{\|v_0\|_{\dot{H}^{\beta-\alpha/2}}^2}{\lambda}+\mathcal{N}C_{FS}(\beta,\alpha,\lambda,\nu),
\end{eqnarray*}
$$
\int_t^{t+1}\|v(s)\|_{\dot{H}^{\beta-\alpha/2}}^{\frac{2+2\delta}{\delta}}ds\leq \left(\|v_0\|_{\dot{H}^{\beta-\alpha/2}}^2+\mathcal{N}C_{FS}(\beta,\alpha,\lambda,\nu)\right)^{\frac{2+2\delta}{2\delta}}.
$$
Notice that if
$$
t\geq t_0=\max\left\{0,\frac{1}{-\lambda}\log\left(\frac{\mathcal{N}C_{FS}(\beta,\alpha,\lambda,\nu)}{\frac{\|v_0\|_{\dot{H}^{\beta-\alpha/2}}^2}{\lambda}(1-e^{-\lambda})}\right)\right\},
$$
we have an inequality that is independent of $v_0$:
$$
\|v(t)\|_{\dot{H}^{\beta-\alpha/2}}^2\leq 2\mathcal{N}C_{FS}(\beta,\alpha,\lambda,\nu).
$$
Then, from \eqref{uL2}, we obtain
\begin{align*}
\frac{d}{dt}\|u(t)\|_{L^2}^2+\frac{\mu}{2}\|u(t)\|_{\dot{H}^{\alpha/2}}^{2}& \leq r\|u(t)\|_{L^2}^2+\frac{(C_{KP}(\alpha)\mathcal{N})^2}{\mu}\|v(t)\|_{\dot{H}^{\beta-\alpha/2}}^{2}\\
&\quad +\frac{(C_{KP}(\alpha)C_{GN}(\alpha))^{\frac{2+2\delta}{\delta}}4\mathcal{N}^{2}\|v(t)\|_{\dot{H}^{\beta-\alpha/2}}^{\frac{2+2\delta}{\delta}}}{\mu}.
\end{align*}
Due to Lemma \ref{L1normr}, we obtain
$$
\int_t^{t+1}\|u(s)\|_{L^2}^2\leq 3\mathcal{N}.
$$
Using Uniform Gronwall estimate (Lemma \ref{UGL}) and the previous inequality, we have that
$$
\|u(t+1)\|_{L^2}^2\leq S(L^2)\quad \forall\,t\geq t_0,
$$
where $S(L^2)$ is defined in \eqref{SL2}. Using the previous inequality we also obtain
$$
\int_t^{t+1}\|u(s)\|^2_{\dot{H}^{\alpha/2}}ds\leq \frac{2(S(L^2)+S(L^2)e^{-r})}{\mu},\quad \forall\,t\geq t_0+1.
$$
Let us consider $\alpha<2$ (the case $\alpha=2$ can be done straightforwardly). We look for a commutator-type structure in the nonlinearity:
\begin{eqnarray*}
\frac{1}{2}\frac{d}{dt}\|u(t)\|_{\dot{H}^{\alpha/2}}^2&=&-\mu\int_\TT|\Lambda^{\alpha}u|^2dx+
\int_{\TT}u\Lambda^\beta v\Lambda^{\alpha}u dx\\
&&+r\int_\TT|\Lambda^{\alpha/2}u|^2dx-r\int_{\TT}u^2\Lambda^{\alpha}udx\\
&&+\int_{\TT}\left[\Lambda^{\alpha/2},\Lambda^{\beta-1}H v\right]\pax u\Lambda^{\alpha/2}u dx\\
&&+\int_{\TT}\Lambda^{\beta-1}H v\frac{\pax(\Lambda^{\alpha/2}u)^2}{2}dx.
\end{eqnarray*}
We estimate
$$
I_1=\int_{\TT}u\Lambda^\beta v\Lambda^{\alpha}u dx\leq \frac{2}{\mu}\|u\|_{L^\infty}^2\|v\|_{\dot{H}^\beta}^2+\frac{\mu}{8}\|u\|_{\dot{H}^\alpha}^2,
$$
$$
I_2=r\int_{\TT}u^2\Lambda^{\alpha}udx\leq \frac{2r^2}{\mu}\|u\|_{L^\infty}^2\|u\|_{L^2}^2+\frac{\mu}{8}\|u\|_{\dot{H}^\alpha}^2,
$$
$$
I_3=\int_{\TT}\Lambda^{\beta-1}H v\frac{\pax(\Lambda^{\alpha/2}u)^2}{2}dx\leq C_I(\alpha)\|v\|_{\dot{H}^\beta}\|u\|_{\dot{H}^{\alpha/2}}\|u\|_{\dot{H}^{\alpha}},
$$
where we used that, for a zero-mean value $f=\Lambda^{\alpha/2}u$, the following inequality holds
\begin{equation}\label{CI}
\|f\|_{L^4}^2\leq C_I(\alpha)\|f\|_{L^2}\|f\|_{\dot{H}^{\alpha/2}}.
\end{equation}
The yet untouched term is
\[
I_4=\int_{\TT}\left[\Lambda^{\alpha/2},\Lambda^{\beta-1}H v\right]\pax u\Lambda^{\alpha/2}u dx
\leq \left\|\left[\Lambda^{\alpha/2},\Lambda^{\beta-1}H v\right]\pax u\right\|_{L^2}\|u\|_{\dot{H}^{\alpha/2}}.
\]
The Kenig-Ponce-Vega estimate (see Lemma \ref{lemaaux2}) gives
\begin{align}
\|[\Lambda^{\alpha/2},\Lambda^{\beta-1}H v]\pax u\|_{L^2} &\leq C_{KPV}(\alpha)\left(\|\pax u\|_{L^{2+\frac{2\alpha-2}{2-\alpha}}}\|\Lambda^{\beta-1}Hv\|_{W^{\frac{\alpha}{2}, \frac{2}{\alpha-1}}}\right.\nonumber\\
&\quad \left.+\| \pax u \|_{W^{\frac{\alpha}{2}-1,\infty}}\|v\|_{\dot{H}^{\beta}}\right)\label{C_KPV}.
\end{align}
Since both $\pax u$ and $Hv$ have zero-mean, inequalities \eqref{C_SE2} yield
$$
I_4\leq C_{KPV}(\alpha)\|u\|_{\dot{H}^{\alpha/2}}\|u\|_{\dot{H}^{\alpha}}\|v\|_{\dot{H}^{\beta}}\left(C^3_{SE}(\alpha)C^1_{SE}(\alpha)+C^4_{SE}(\alpha)\right).
$$
Young's inequality implies
\begin{eqnarray*}
I_4&\leq& \|v\|_{\dot{H}^{\beta}}^2\|u\|_{\dot{H}^{\alpha/2}}^2\frac{\left[C_{KPV}(\alpha)(C^3_{SE}(\alpha)C^1_{SE}(\alpha)+C^4_{SE}(\alpha))\right]^2}{\mu}+\frac{\mu}{4}\|u\|_{\dot{H}^{\alpha}}^2.
\end{eqnarray*}
Collecting every estimate, we have eventually that
\begin{equation}\label{eq:rev1}
\frac{d}{dt}\|u\|_{\dot{H}^{\alpha/2}}^2+\mu\|u\|_{\dot{H}^\alpha}^2\leq2\|u\|_{{H}^{\alpha/2}}^2g(t),
\end{equation}
with
\begin{multline*}
g(t)=r+\frac{2C^2_{SE}(\alpha)r^2}{\mu}\|u\|_{L^2}^2+\frac{\|v\|_{\dot{H}^\beta}}{2}\\
+\frac{2C^2_{SE}(\alpha)+(C_I(\alpha))^2}{\mu}\|v\|_{\dot{H}^\beta}^2\\
+\frac{\left[C_{KPV}(\alpha)(C^3_{SE}(\alpha)C^1_{SE}(\alpha)+C^4_{SE}(\alpha))\right]^2}{\mu}\|v\|_{\dot{H}^\beta}^2.
\end{multline*}
We get control over the full ${H}^{\alpha/2}$ norm by testing equation for $u$ with $u$. A straightforward computation there together with \eqref{eq:rev1} gives
\[
\frac{d}{dt}\|u\|_{{H}^{\alpha/2}}^2+\mu\|u\|_{\dot{H}^\alpha}^2\leq2\|u\|_{{H}^{\alpha/2}}^2\left(g(t) + \frac{1}{2} +r\right).
\]
Testing the equation for $v$ with $\Lambda^{\beta} v$ and using $\beta\geq\alpha$, we have
$$
\int_t^{t+1}\|v(s)\|_{\dot{H}^{\beta}}^2ds\leq \frac{\mathcal{N}}{\nu}\left(\frac{3}{\nu}+2C_{FS}(\beta,\alpha,\lambda,\nu)\right)\quad \forall\, t\geq t_0,
$$
so, if $t\geq t_0$,
\begin{multline*}
\int_t^{t+1}g(s)ds\leq r+\frac{r^2C^2_{SE}(\alpha)}{\mu}6\mathcal{N}+\frac{\mathcal{N}}{\nu}\left(\frac{3}{\nu}+2C_{FS}(\beta,\alpha,\lambda,\nu)\right)\\
\times\left(\frac{1}{2}+\frac{2C^2_{SE}(\alpha)+(C_I(\alpha))^2}{\mu} \frac{\left[C_{KPV}(\alpha)(C^3_{SE}(\alpha)C^1_{SE}(\alpha)+C^4_{SE}(\alpha))\right]^2}{\mu}\right).
\end{multline*}
Using Lemma \ref{UGL}, we obtain
$$
\|u(t+1)\|_{\dot{H}^{\alpha/2}}^2\leq S(\dot{H}^{\alpha/2})\quad \forall\,t\geq t_0+1,
$$
$$
\int_{t}^{t+1}\|u(s)\|_{\dot{H}^\alpha}^2ds\leq \frac{2}{\mu}S(\dot{H}^{\alpha/2})\left(\frac{1}{2}+\int_t^{t+1}g(s)ds\right)\quad \forall\,t\geq t_0+2,
$$
with
$$
S(\dot{H}^{\alpha/2})=\frac{2(S(L^2)+S(L^2)e^{-r})}{\mu}e^{2\int_t^{t+1}g(s)ds}.
$$

Hence we have obtained the absorbing set in $H^{k\alpha}$ with $k=1$.

Due to the linear character of the equation for $v$, we have that
$$
\|v(t)\|_{\dot{H}^{\beta}}^2\leq \|v_0\|_{\dot{H}^{\beta}}^2e^{-\lambda t}+\mathcal{M}(\dot{H}^{\alpha/2})C_{FS}(\beta,\alpha,\lambda,\nu),
$$
where, for a given space $X$, we set
$$
\mathcal{M}(X)=\sqrt{2\pi}\max\left\{\max_{0\leq s\leq T^*}\|u(s)\|_{X},S(X)\right\},
$$
for a $T^*>>1$ that will be fixed later. We remark that $\|\Lambda^{\alpha/2}u(t)\|_{L^1}\leq \mathcal{M}(\dot{H}^{\alpha/2})$.

Now we can continue in the same way using induction. Once we have the absorbing set for $u$ in $L^\infty([0,\infty],H^{k\alpha/2}(\TT))$ ($k\geq 1$) and the bound $u\in L^2([t,t+1],H^{(k+1)\alpha/2}(\TT))$, we can ensure that $v\in L^\infty([0,\infty],H^{\beta+(k-1)\alpha/2}(\TT))$ and $v\in  L^2([t,t+1],H^{\beta+k\alpha/2}(\TT))$. Now we test the equation for $u$ against $\Lambda^{(k+1)\alpha}u$ to get
\begin{eqnarray*}
\frac{1}{2}\frac{d}{dt}\|u(t)\|_{\dot{H}^{(k+1)\alpha/2}}^2&=&-\mu\int_\TT|\Lambda^{\left(\frac{k}{2}+1\right)\alpha}u|^2dx+r\int_\TT|\Lambda^{(k+1)\alpha/2}u|^2dx\\
&&+\int_{\TT}\Lambda^{k\alpha/2}(u\Lambda^\beta v)\Lambda^{(k/2+1)\alpha}u dx\\
&&-r\int_{\TT}\Lambda^{k\alpha/2}(u^2)\Lambda^{(k/2+1)\alpha}udx\\
&&+\int_{\TT}\left[\Lambda^{(k+1)\alpha/2},\Lambda^{\beta-1}H v\right]\pax u\Lambda^{(k+1)\alpha/2}u dx\\
&&+\int_{\TT}\Lambda^{\beta-1}H v\frac{\pax(\Lambda^{(k+1)\alpha/2}u)^2}{2}dx.
\end{eqnarray*}
To conclude the existence of $S\left(\frac{(k+1)\alpha}{2}\right)$ we use Lemma \ref{lemaaux2} and the same ideas. Finally, notice that at each iteration step we have to add $1$ to the initial value $t_0$. Consequently, we need to take $T^*=T^*(k)$ large enough to reach $H^{k\alpha}(\TT)$. For instance, to reach $H^3$, $T^{*}=t_0+10$ suffices.

\section{Proofs concerning the smoothing effect}\label{S5}
We begin with proving our main result concerning the smoothing effect of the system \eqref{eqa1}-\eqref{eqa2}

\begin{proof}[Proof of Theorem \ref{smoothingeffect}]
Recall that for $t>0$ the finiteness of the Hardy-Sobolev norm \eqref{9} implies the analyticity on the real line. We define $z=x\pm i\omega t$. In this complex strip the extended system reads
\begin{eqnarray}\label{eqa1complex}
\pat u(z) & = & -\mu\Lambda^\alpha u(z)+\pax\cdot(u(z)\Lambda^{\beta-1} H v(z))\\\nonumber
&& +ru(z)(1-u(z)),\\ 
\label{eqa2complex}
\pat v(z) & = & -\nu\Lambda^\beta v(z)-\lambda v(z)+u(z).
\end{eqnarray}
We are going to perform new energy estimates in the Hardy-Sobolev space \eqref{8} for an appropriate value of $\omega$. Notice that, as the functions $u$ and $v$ are complex for complex arguments, the integration by parts is a delicate matter for some terms. Consequently, there are several new terms appearing that are not present in the real case. 

We deal first with the case $\alpha,\beta>1$. At the end of the proof we will explain how to cover the extreme case $\alpha=\beta=1$. We restrict here to formal estimates, as their rigorization is analogous to that for the real case.

Let us start with the estimates for the equation \eqref{eqa2complex}. Using $\int f\bar{g}=\overline{\int \bar{f}g}$, we have
$$
\frac{d}{dt}\|v\|_{L^2(\Ss_{\omega})}^2=2\re\int_\TT\bar{v}(z)\left(\pat v(z)\pm i\omega\pax v(z)\right)dx.
$$
Using Plancherel's Theorem, we have
$$
\re\int_\TT\bar{v}(z)(-\nu\Lambda^{\beta}v(z)-\lambda v(z))dx=-\nu\|v\|^2_{\dot{H}^{\beta/2}(\Ss_{\omega})}-\lambda\|v\|^2_{L^2(\Ss_{\omega})}\leq 0.
$$
Consequently, using \eqref{eqa2complex},
$$
\frac{1}{2}\frac{d}{dt}\|v\|_{L^2(\Ss_{\omega})}^2\leq \|v\|_{L^2(\Ss_{\omega})}\left(\omega\|v\|_{H^1(\Ss_{\omega})}+\|u\|_{L^2(\Ss_{\omega})}\right).
$$
Taking 4 derivatives of the equation \eqref{eqa2complex} and testing against $\pax^4v$, we obtain
$$
\frac{1}{2}\frac{d}{dt}\|v\|_{\dot{H}^4(\Ss_{\omega})}^2=\re\int_\TT\pax^4\bar{v}(z)\left(\pat \pax^4v(z)\pm i\omega\pax^5 v(z)\right)dx,
$$
\begin{align*}
\frac{1}{2}\frac{d}{dt}\|v\|_{H^4(\Ss_{\omega})}^2&\leq -\nu\|v\|_{\dot{H}^{4+\beta/2}(\Ss_{\omega})}^2+3\omega\|v\|_{\dot{H}^{4+1/2}(\Ss_{\omega})}^2+\frac{\mu}{4}\|u\|_{\dot{H}^{3+\alpha/2}(\Ss_{\omega})}^{2}\\
&\quad+\frac{(C_{SI}(\alpha))^2}{\omega}\frac{\alpha-1}{\alpha}\left(\frac{\frac{\mu}{4}\omega \alpha}{(C_{SI}(\alpha))^2}\right)^{-1/(\alpha-1)} \|u\|_{\dot{H}^{3}(\Ss_{\omega})}^{2}\\
&\quad+\left(\omega+\frac{1}{2}\right)\|v\|_{H^4(\Ss_{\omega})}^2+\frac{\|u\|_{H^3(\Ss_{\omega})}^2}{2}.
\end{align*}

Now we proceed with the equation for $u$. The lower order term can be bounded easily as follows
\begin{align*}
\frac{1}{2}\frac{d}{dt}\|u\|_{L^2(\Ss_{\omega})}^2&=\re\int_\TT\bar{u}(z)\left(\pat u(z)\pm i\omega\pax u(z)\right)dx\\
&\leq \|u\|_{L^2(\Ss_\omega)}\left(\omega\|u\|_{H^1(\Ss_\omega)}+\|u\|_{L^\infty(\Ss_\omega)}\|v\|_{\dot{H}^\beta(\Ss_\omega)}\right.\\
&\quad +r\|u\|_{L^2(\Ss_\omega)}+r\|u\|_{L^2(\Ss_\omega)}\|u\|_{L^\infty(\Ss_\omega)}\\
&\quad \left.+\|u\|_{H^1(\Ss_\omega)}\|\Lambda^{\beta-1}H v\|_{L^\infty(\Ss_\omega)}\right).
\end{align*}
The higher order seminorm contributes with
$$
\frac{1}{2}\frac{d}{dt}\|u\|_{\dot{H}^3(\Ss_{\omega})}^2=\re\int_\TT\pax^3\bar{u}(z)\left(\pat \pax^3 u(z)\pm i\omega\pax^4 u(z)\right)dx.
$$
For the sake of brevity, let  us focus now on the most singular terms. They are
$$
L_1=\frac{1}{2}\re\int_\TT \Lambda^{\beta-1}H v(z)\pax |\pax ^3 u(z)|^2dx\leq \frac{1}{2}\|u\|_{\dot{H}^3(\Ss_\omega)}^2\|\Lambda^{\beta}v\|_{L^\infty(\Ss_\omega)},
$$
$$
L_2=\int_\TT \im\Lambda^{\beta-1}H v(z)\pax ^3 \re\, u(z)\pax ^4 \im\, u(z)dx,
$$
$$
L_3=-\int_\TT \im\Lambda^{\beta-1}H v(z)\pax ^3 \im\, u(z)\pax ^4 \re\, u(z)dx\leq L_2+\|u\|_{\dot{H}^3(\Ss_\omega)}^2\|\Lambda^{\beta}v\|_{L^\infty(\Ss_\omega)}.
$$
Using $\Lambda H=-\pax $ and the self-adjointness of $\Lambda^s$, we find a commutator and estimate
\begin{multline*}
L_2\leq \left\|\left[\Lambda^{0.5},\im\Lambda^{\beta-1}H v\right]\pax ^3 \re\, u\right\|_{L^2(\Ss_\omega)}\|u\|_{\dot{H}^{3.5}(\Ss_\omega)}\\
+\|\im\Lambda^{\beta-1}H v\|_{L^\infty(\Ss_\omega)}\|u\|_{\dot{H}^{3.5}(\Ss_\omega)}^2.
\end{multline*}
We use Lemma \ref{lemaaux2} to obtain the commutator estimate
$$
\left\|\left[\Lambda^{0.5},\im\Lambda^{\beta-1}H v\right]\pax ^3 \re\, u\right\|_{L^2(\Ss_\omega)}\leq C_{KPV}^2\|\Lambda^\beta v\|_{L^\infty(\Ss_\omega)}\|u\|_{\dot{H}^{3}(\Ss_\omega)}.
$$
Putting all the estimates together and using Sobolev embedding together with $3+\beta\leq 4+\beta/2$, we have that
\begin{eqnarray*}
\frac{d}{dt}\|u\|_{H^3(\Ss_{\omega})}^2&\leq& 2\|u\|_{H^3(\Ss_\omega)}^2\left(\omega+17.5C^2_{SE}(\alpha)\|v\|_{\dot{H}^4(\Ss_\omega)}+2r\right.\\
&&\left.+2C^2_{SE}(\alpha)\|v\|_{\dot{H}^{4+\beta/2}(\Ss_\omega)}\right)+\|u\|_{H^3(\Ss_\omega)}^39rC^2_{SE}(\alpha)\\
&&-2\mu\|u\|_{\dot{H}^{3+\alpha/2}(\Ss_{\omega})}^2+4\omega\|u\|_{\dot{H}^{3.5}(\Ss_{\omega})}^2\\
&&+2\|\im\,\Lambda^{\beta-1}H v\|_{L^\infty(\Ss_\omega)}\|u\|_{\dot{H}^{3.5}(\Ss_\omega)}^2\\
&&+2\|u\|_{\dot{H}^{3.5}(\Ss_\omega)}\|u\|_{\dot{H}^{3}(\Ss_\omega)}C_{KPV}^2C_{SE}^2(\alpha)\|v\|_{H^4(\Ss_\omega)}.
\end{eqnarray*}
Notice that via Poincar\'e inequality holds
$$
\|\im\,\Lambda^{\beta-1}H v\|_{L^\infty}\leq \sqrt{2}\|\im\,v\|_{\dot{H}^\beta}\leq \sqrt{2}\|\im\,v\|_{\dot{H}^2}.
$$
Let us define the energy
\begin{multline*}
E(t)=1+\|u\|_{H^{3}(\Ss_\omega)}^2+\|v\|_{H^{4}(\Ss_\omega)}^2+\frac{1}{\frac{\mu}{4}-\sqrt{2}\|\im\,v\|_{\dot{H}^2}}\\
+\left\|\frac{1}{2\|u_0\|_{L^\infty(\TT)}-|u(z)|^2}\right\|_{L^\infty}+\left\|\frac{1}{2\|v_0\|_{L^\infty(\TT)}-|v(z)|^2}\right\|_{L^\infty}.
\end{multline*}
We obtain (see \cite{CGO,GH,GNO} for further details)
$$
\frac{d}{dt}\left\|\frac{1}{2\|u_0\|_{L^\infty(\TT)}-|u(z)|^2}\right\|_{L^\infty}\leq (\mu+2+2r)C_{SE}^2(\alpha)(E(t))^4.
$$
In the same way
$$
\frac{d}{dt}\left\|\frac{1}{2\|v_0\|_{L^\infty(\TT)}-|v(z)|^2}\right\|_{L^\infty}\leq (\nu+\lambda+1)C_{SE}^2(\alpha)(E(t))^3.
$$

Thus, putting all the estimates together, we get
\begin{eqnarray*}
\frac{d}{dt}E(t)&\leq& -2\nu\|v\|_{\dot{H}^{4+\beta/2}(\Ss_{\omega})}^2+6\omega\|v\|_{\dot{H}^{4+1/2}(\Ss_{\omega})}^2+2 \left(\frac{\mu}{4}-\mu \right)\|u\|_{\dot{H}^{3+\alpha/2}(\Ss_{\omega})}^{2}\\
&&+2\bigg{[}\frac{(C_{SI}(\alpha))^2}{\omega}\frac{\alpha-1}{\alpha}\left(\frac{\omega^2 \alpha}{(C_{SI}(\alpha))^2}\right)^{-1/(\alpha-1)} +2\omega\\
&&+1.25 +2r+(17.5C_{SE}^2(\alpha))^2+\frac{(9rC_{SE}^2(\alpha))^2}{2}\bigg{]}E(t)\\
&&+(E(t))^2\left(1.5+\frac{2(C_{KPV}^2C_{SE}^2(\alpha))^2}{\mu}\right)\\
&&+(1+(\mu+2+2r+\nu+\lambda+1)C_{SE}^2(\alpha))(E(t))^4\\
&&+\left(\frac{\mu}{2}+(\sqrt{2})^3\|\im\,v\|_{\dot{H}^2}+4\omega\right)\|u\|_{\dot{H}^{3.5}(\Ss_\omega)}^2.
\end{eqnarray*}
Now observe that, as long as $E(t)<\infty,$ we have $\frac{\mu}{4}-\sqrt{2}\|\im\,v\|_{\dot{H}^2}>0$ and, using Poincar\'e inequality if needed, we obtain
\begin{eqnarray*}
\frac{d}{dt}E(t)&\leq& -2\nu\|v\|_{\dot{H}^{4+\beta/2}(\Ss_{\omega})}^2+6\omega\|v\|_{\dot{H}^{4+1/2}(\Ss_{\omega})}^2-\frac{\mu}{2}\|u\|_{\dot{H}^{3+\alpha/2}(\Ss_{\omega})}^{2}\\
&&+2\bigg{[}\frac{(C_{SI}(\alpha))^2}{\omega}\frac{\alpha-1}{\alpha}\left(\frac{\omega^2 \alpha}{(C_{SI}(\alpha))^2}\right)^{-1/(\alpha-1)} +2\omega\\
&&+1.25 +2r+(17.5C_{SE}^2(\alpha))^2+\frac{(9rC_{SE}^2(\alpha))^2}{2}\bigg{]}E(t)\\
&&+(E(t))^2\left(1.5+\frac{2(C_{KPV}^2C_{SE}^2(\alpha))^2}{\mu}\right)\\
&&+(1+(\mu+2+2r+\nu+\lambda+1)C_{SE}^2(\alpha))(E(t))^4\\
&&+4\omega\|u\|_{\dot{H}^{3.5}(\Ss_\omega)}^2.
\end{eqnarray*}
For $\alpha,\beta>1$ we have by Plancharel Theorem
$$
4\omega\|u\|_{\dot{H}^{3.5}(\Ss_\omega)}^2-\frac{\mu}{2}\|u\|_{\dot{H}^{3+\alpha/2}(\Ss_\omega)}^2\leq \mathcal{C}_1\|u\|_{\dot{H}^{3}(\Ss_\omega)}^2,
$$
$$
3\omega\|v\|_{\dot{H}^{4.5}(\Ss_\omega)}^2-\nu\|v\|_{\dot{H}^{4+\beta/2}(\Ss_\omega)}^2\leq \mathcal{C}_2\|v\|_{\dot{H}^{4}(\Ss_\omega)}^2,
$$
with $\mathcal{C}_1,\mathcal{C}_2$ given by \eqref{C1},\eqref{C2}. Consequently, we can choose any positive value for $\omega>0$ and we have the inequality
$$
\frac{d}{dt}E(t)\leq \mathcal{K}_1(E(t))^4,
$$
with $\mathcal{K}_1,\mathcal{C}_i$ defined in \eqref{C1},\eqref{C2} and \eqref{K1}. Solving this ODI, we obtain
$$
E(t)\leq \frac{1}{\sqrt[3]{\frac{1}{1+\|u_0\|_{H^{3}(\TT)}^2+\|v_0\|_{H^{4}(\TT)}^2}-3t
\mathcal{K}_1}},
$$
and, using \eqref{K1}, we conclude that $(u,v)$ are analytic functions at least for time
$$
\tilde{T}=\frac{1}{3\mathcal{K}_1(1+\|u_0\|_{H^{3}(\TT)}^2+\|v_0\|_{H^{4}(\TT)}^2)}
$$
Notice that in the extreme cases $\min\{\alpha,\beta\}=1$, we can take
$$
0<\omega\leq \omega_0,
$$
(with $\omega_0$ defined in \eqref{omega0}) to obtain the inequality
\begin{equation}\label{10}
\frac{d}{dt}E(t)\leq \mathcal{K}_2(E(t))^4,
\end{equation}
with $\mathcal{K}_2$ given by \eqref{K2}. From the inequality \eqref{10} we obtain
$$
E(t)\leq \frac{1}{\sqrt[3]{\frac{1}{1+\|u_0\|_{H^{3}(\TT)}^2+\|v_0\|_{H^{4}(\TT)}^2}-3t\mathcal{K}_2}},
$$
and we again conclude that the solution $(u,v)$ is analytic for time $t<\tilde{T}$ with
$$
\tilde{T}=\frac{1}{3\mathcal{K}_2(1+\|u_0\|_{H^{3}(\TT)}^2+\|v_0\|_{H^{4}(\TT)}^2)}.
$$
\end{proof}

\begin{remark}
A similar Theorem holds for the parabolic-elliptic system (\eqref{eqa1}-\eqref{eqa2} with $\tau=0$). We refer the reader to \cite{AGM} for details on how to adapt the proof.
\end{remark}

\begin{proof}[Proof of Corollary \ref{coro1}]
The proof of Corollary \ref{coro1} is obtained by a standard continuation argument. First notice that the solution $(u(t),v(t))\in H^3(\TT)\times H^4(\TT)$ globally and it is unique. In particular, at $t=\tilde{T}$, we can restart the evolution with initial data $(u_0^1,v_0^1)=(u(\tilde{T}),v(\tilde{T}))$. The initial data may not be analytic, but there exists a $\delta$ small enough so $(u^1(t),v^1(t))=(u(\tilde{T}+t),v(\tilde{T}+t))$ is analytic for $0<t<\delta$. As we can find such a positive $\delta$ for every initial data, we conclude. In other words, if we can not find such a positive $\delta$, it is because $(u_0^n,v_0^n) \notin H^3(\TT)\times H^4(\TT)$, and this is a contradiction. \end{proof}
For the proof of Corollary \ref{illposed} we refer to \cite{AGM}, \cite{CGO}, \cite{GH}.

\section{Proof of Theorem \ref{attractor}: Existence of the attractor}\label{S6}
Here we prove the existence and some properties of the attractor. First we need a definition coming from dynamical systems (see Temam \cite{temambook}).

\begin{defi}
The solution operator $S(t)(u_0,v_0)=(u(t,x),v(t,x))$ defines a compact semiflow in $H^{3\alpha}(\TT)\times H^{\beta/2+3\alpha}(\TT)$ if, for every $(u_0,v_0)\in H^{3\alpha}(\TT)\times H^{\beta/2+3\alpha}(\TT)$, the following statements hold:
\begin{itemize}
\item $S(0)(u_0,v_0)=(u_0,v_0)$.
\item for all $t,s,u_0,v_0$, the semigroup property hold, \emph{i.e.}
$$
S(t+s)(u_0,v_0)=S(t)S(s)(u_0,v_0)=S(s)S(t)(u_0,v_0).
$$
\item For every $t>0$
$$
S(t)(\cdot,\cdot): \; H^{3\alpha}(\TT)\times H^{\beta/2+3\alpha}(\TT)\mapsto H^{3\alpha}(\TT)\times H^{\beta/2+3\alpha}(\TT)
$$ 
is continuous.
\item There exists $T^*>0$ such that $S(T^*)$ is a compact operator, \emph{i.e.} for every bounded set $B\subset H^{3\alpha}(\TT)\times H^{\beta/2+3\alpha}(\TT)$, $S(T^*)B\subset H^{3\alpha}(\TT)\times H^{\beta/2+3\alpha}(\TT)$ is a compact set.
\end{itemize}
\end{defi}

We have
\begin{lem}
Let $T>0$, $8/7\leq\alpha\leq\beta\leq2$, $\mu,\nu,\lambda,r>0$, $(u_0,v_0)\in H^{3\alpha}\times H^{\beta/2+3\alpha}$. Then $S(\cdot)(u_0,v_0)=(u(\cdot)),v(\cdot)))\in C([0,T],H^{3\alpha}(\TT)\times H^{\beta/2+3\alpha}(\TT))$ for every initial data and it defines a compact semiflow in $H^{3\alpha}\times H^{\beta/2+3\alpha}$. 
\end{lem}
\begin{proof} As in Theorem \ref{globalpp1} we have
\begin{equation}\label{1}
\Lambda^{3\alpha}u\in L^\infty([0,T],L^2)\cap L^2([0,T],H^{\alpha/2}),
\end{equation}
\begin{equation}\label{2}
\Lambda^{3\alpha+\beta/2}v\in L^\infty([0,T],L^2)\cap L^2([0,T],H^{\beta/2}).
\end{equation}
We have to prove that 
$$
\pat\Lambda^{3\alpha}u\in L^2([0,T],H^{-\alpha/2}),\pat\Lambda^{3\alpha+\beta/2}v\in L^2([0,T],H^{-\beta/2}).
$$
By duality and the Kato-Ponce inequality, we have
\begin{eqnarray*}
\|\pat \Lambda^{3\alpha}u\|_{\dot{H}^{-\alpha/2}}&=&\sup_{\|\phi\|_{H^{\alpha/2}}\leq 1}\left|\int_\TT\pat \Lambda^{3\alpha}u \phi dx\right|\\
&\leq&\mu\|\Lambda^{3.5\alpha}u\|_{L^2}+\|\Lambda^{2.5\alpha}(u\Lambda^{\beta}v)\|_{L^2}+\|\Lambda^{2.5\alpha}(\pax u \Lambda^{\beta-1}Hv)\|_{L^2}\\
&&+r\|\Lambda^{2.5\alpha}u\|_{L^2}+r\|\Lambda^{2.5\alpha}(u^2)\|_{L^2}\\
&\leq& C\left(\|\Lambda^{3.5\alpha}u\|_{L^2}+\|\Lambda^{2.5\alpha}u\|_{L^2}\|\Lambda^{\beta}v\|_{L^\infty}+\|\Lambda^{\beta+2.5\alpha}v\|_{L^2}\|u\|_{L^\infty}\right.\\
&&+\|\Lambda^{2.5\alpha}\pax u\|_{L^2}\|\Lambda^{\beta-1}Hv\|_{L^\infty}+\|\pax u\|_{L^\infty}\| \Lambda^{\beta+2.5\alpha-1}v\|_{L^2}\\
&&\left.+\|\Lambda^{2.5\alpha}u\|_{L^2}+\|u\|_{L^\infty}\|\Lambda^{2.5\alpha}u\|_{L^2}\right)\\
&\leq& C\left(\|u\|_{\dot{H}^{3.5\alpha}}+\|u\|_{H^{3\alpha}}\|v\|_{H^{3\alpha+\beta/2}}+\|v\|_{\dot{H}^{\beta+3\alpha}}\|u\|_{H^{3\alpha}}\right.\\
&&+\|u\|_{\dot{H}^{3.5\alpha}}\|v\|_{H^{\beta/2+3\alpha}}+\|u\|_{H^{3\alpha}}\|v\|_{H^{\beta+3\alpha}}\\
&&\left.+\|u\|_{H^{3\alpha}}+\|u\|_{H^{3\alpha}}\|u\|_{H^{3\alpha}}\right),
\end{eqnarray*}
and we conclude the bound for $\pat\Lambda^{3\alpha}u$:
\begin{equation}\label{3}
\int_0^T \|\pat \Lambda^{3\alpha}u(s)\|^2_{\dot{H}^{-\alpha/2}}ds\leq C.
\end{equation}
We proceed in the same way to get
\begin{equation}\label{4}
\int_0^T\|\pat \Lambda^{3\alpha+\beta/2}v(s)\|^2_{\dot{H}^{-\alpha/2}}ds\leq C
\end{equation}
The bounds \eqref{3} and \eqref{4} together with \eqref{1} and \eqref{2} imply
$$
u\in C([0,T],\dot{H}^{3\alpha}),\,v\in C([0,T],\dot{H}^{3\alpha+\beta/2}).
$$
To get the full norm we use 
$$
u\in L^\infty([0,T],L^2)\cap L^2([0,T],H^{\alpha/2}), \;
v\in L^\infty([0,T],L^2)\cap L^2([0,T],H^{\beta/2}).
$$
and repeat the argument for 
$$
\pat u\in L^2([0,T],H^{-\alpha/2}),\pat v\in L^2([0,T],H^{-\beta/2}).
$$
This implies
$$
S(\cdot)(u_0,v_0)\in C([0,T], H^{3\alpha}\times H^{\beta/2+3\alpha}).
$$
The semigroup property follows from the uniqueness of the classical solution. Having fixed $s_0$, the continuity of 
$$
S(s_0)(\cdot,\cdot):H^{3\alpha}\times H^{\beta/2+3\alpha}\mapsto H^{3\alpha}\times H^{\beta/2+3\alpha},
$$
can be obtained with the energy estimates. Finally, we use Theorem \ref{smoothingeffect} to get that $S(t)(u_0,v_0)\in H^{3.5\alpha}\times H^{\beta+3\alpha}$ if $t\geq \delta$, for every initial data and $\delta>0$. As in Theorem \ref{globalpp1abs}, we obtain the existence of $T^*$ and a constant $C$ such that
$$
\max_{t\geq T^*}\{\|u(t)\|_{H^{3.5\alpha}}+\|v(t)\|_{H^{\beta+3\alpha}}\}\leq C.
$$
Using the compactness of the embeddings $H^{\epsilon}\hookrightarrow L^2$, we conclude the result.
\end{proof}
\begin{remark}
The restriction $\alpha\leq\beta$ is to get the existence of the absorbing sets when applying Theorem \ref{globalpp1abs}. The restriction $8/7\leq \alpha$ is to get
$$
3\alpha+\beta/2\geq 3.5\alpha\geq 4,
$$
to invoke Theorem \ref{smoothingeffect}.
\end{remark}

\begin{proof}[Proof of Theorem \ref{attractor}]
We can use the previous Lemma together with Theorem \ref{globalpp1abs} and Theorem 1.1 in \cite{temambook} to conclude Theorem \ref{attractor}.
\end{proof}

\section{Proof of Theorem \ref{oscillations}: The number of relative maxima}\label{S8}
Finally, let us provide the proof of Theorem \ref{oscillations}:
\begin{proof}[Proof of Theorem \ref{oscillations}]
Using Theorem \ref{smoothingeffect} for $\tilde{T}/(N-1)<t<\tilde{T}$, $N\geq3$ and $\omega=\omega_0$ defined in \eqref{omega0}, we have that the solutions become analytic in a strip with width at least 
$$
\mathcal{W}=\frac{\omega \tilde{T}}{N}=\frac{\omega}{N}\frac{1+\|u_0\|_{H^{3}(\TT)}^2+\|v_0\|_{H^{4}(\TT)}^2}{3\mathcal{K}},
$$
and $\mathcal{K}$ given by \eqref{K}. We have
$$
\omega\tilde{T}\left(\frac{1}{N-1}-\frac{1}{N}\right)\leq \omega t-\mathcal{W}
$$
and using Cauchy's formula and Hadamard's three lines theorem
$$
\|\pax u\|_{L^\infty(\{|\Im z|\leq\mathcal{W}\})}\leq \frac{N(N-1)\|u\|_{L^\infty(\{|\Im z|\leq\omega t\})}}{\omega \tilde{T}}\leq \frac{\sqrt{2}(N-1)\|u_0\|_{L^\infty(\TT)}}{\mathcal{W}},
$$
$$
\|\pax v\|_{L^\infty(\{|\Im z|\leq\mathcal{W}\})}\leq \frac{N(N-1)\|v\|_{L^\infty(\{|\Im z|\leq\omega t\})}}{\omega \tilde{T}}\leq \frac{\sqrt{2}(N-1)\|v_0\|_{L^\infty(\TT)}}{\mathcal{W}}.
$$
Using Lemma \ref{grujic}, we have that for any $\epsilon>0$, $0<\tilde{T}/(N-1)< t<\tilde{T}$, $\TT=I^u_\epsilon\cup R^u_\epsilon=I^v_\epsilon\cup R^v_\epsilon$, where $I^u_\epsilon,I^v_\epsilon$ are the union of at most $[\frac{4\pi}{\mathcal{W}}]$ intervals open in $\TT$, and
\begin{itemize}
\item $|\pax u(x)| \leq \epsilon, \text{ for all }x\in I^u_\epsilon,$
\item $\text{card}\{x \in R^u_\epsilon : \pax u(x)=0\}\leq
\frac{2}{\log 2}\frac{2\pi}{\mathcal{W}}\log\left(\frac{\sqrt{2}(N-1)\|u_0\|_{L^\infty(\TT)}}{\mathcal{W}\epsilon}\right),$
\item $|\pax v(x)| \leq \epsilon, \text{ for all }x\in I^u_\epsilon,$
\item $\text{card}\{x \in R^v_\epsilon : \pax v(x)=0\}\leq
\frac{2}{\log 2}\frac{2\pi}{\mathcal{W}}\log\left(\frac{\sqrt{2}(N-1)\|v_0\|_{L^\infty(\TT)}}{\mathcal{W}\epsilon}\right).$
\end{itemize}
\end{proof}

\begin{proof}[Proof of Corollary \ref{oscillationsattractor}]
Notice that in the case $\min\{\alpha,\beta\}>1$, we are free to choose
$$
\omega=\frac{N}{1+\|u_0\|_{H^{3}(\TT)}^2+\|v_0\|_{H^{4}(\TT)}^2},
$$
and we can improve the statement in Theorem \ref{oscillations}. Applying Lemma \ref{grujic}, we have that for any $\epsilon>0$, $\TT=I^u_\epsilon\cup R^u_\epsilon=I^v_\epsilon\cup R^v_\epsilon$, with $I^u_\epsilon,I^v_\epsilon$ are the union of at most $[12\pi\mathcal{K}_1]$ intervals open in $\TT$, and
\begin{itemize}
\item $|\pax u(x)| \leq \epsilon, \text{ for all }x\in I^u_\epsilon,$
\item $\text{card}\{x \in R^u_\epsilon : \pax u(x)=0\}\leq
\frac{12\pi\mathcal{K}_1}{\log 2}\log\left(\frac{\sqrt{18}\mathcal{K}_1(N-1)\|u_0\|_{L^\infty(\TT)}}{\epsilon}\right),$
\item $|\pax v(x)| \leq \epsilon, \text{ for all }x\in I^u_\epsilon,$
\item $\text{card}\{x \in R^v_\epsilon : \pax v(x)=0\}\leq
\frac{12\pi\mathcal{K}_1}{\log 2}\log\left(\frac{\sqrt{18}\mathcal{K}_1(N-1)\|v_0\|_{L^\infty(\TT)}}{\epsilon}\right).$
\end{itemize}
We are interested in the points of maximum such that they are close to regions with derivative bigger than one (the so-called peaks). Consequently, we take $\epsilon=1$ and $N=3$. Finally, notice that, in the attractor, we have
$$
\|u(t)\|_{L^\infty}\leq C^2_{SE}(\alpha)S(H^{\alpha/2})
$$
to conclude the result.
\end{proof}
The proof of Corollary \ref{oscillationsattractor2} follows from the same ideas as before.

\section{Numerical simulations}\label{S7}
\subsection{Algorithm}
The dynamics differs substantially depending on the value of the parameters presents in the problem. Hence, in order to reduce the number of parameters, let us consider the one-parameter problem
\begin{eqnarray}\label{eqa1num}
\pat u & = & -\Lambda^\alpha u+\chi\pax\cdot(u\Lambda^{\beta-1} H v) +u(1-u)\\ 
\label{eqa2num}
\pat v & = & -\Lambda^\beta v- v+u, 
\end{eqnarray}
with $\chi>0$. 

To simulate this problem we use the well-known Fourier-collocation method. First we discretize the spatial domain using $N$ uniformly distributed points. Notice that the numerical solution $(u_N,v_N)$ will have $N$ points for $u_N$ and $N$ points for $v_N$.

We use the Fast Fourier Transform (FFT) to change to the frequency space. There the differential operators and the Hilbert transform act as multipliers. Indeed, if we denote the FFT using $\text{FFT}(\cdot)$, we have
$$
\text{FFT}(\Lambda^\gamma u_N)=|\xi|^\gamma \text{FFT}(u_N),\text{FFT}(\Lambda^{\beta-1} H u_N)=-i\xi|\xi|^{\beta-2}\text{FFT}(u_N).
$$

To compute the nonlinear term we use the Inverse Fast Fourier Transform (IFFT) to change back to the physical space. We multiply there appropriately and then we go back to the frequency space using FFT. In particular, writing $\text{IFFT}(\cdot)$ for the IFFT we have that the nonlinearity can be written as
$$
i\xi\text{FFT}\left(u_N\text{IFFT}(-i\xi|\xi|^{\beta-2}\hat{u_N})\right).
$$ 

In this way we can write our problem as a ordinary differential equation in the frequency space. Now we can advance up to time $T$ using our favorite numerical integrator. In particular, we choose the function \texttt{ode45} in Matlab.

\subsection{Results}
\subsubsection{$\alpha=1,\beta=1$} 
First, we study the case $\alpha=1$ and $\beta=1$. For high values of $\chi$, we observe the same chaotic behavior as in \cite{Hillen4}. The homogeneous steady state is unstable and a number of peaks eventually emerge and merge with other peaks (see Figures \ref{emerging}, \ref{merging} and \ref{peaksdynamics2}).
\begin{figure}[t!]
    \centering
    \includegraphics[scale=0.28]{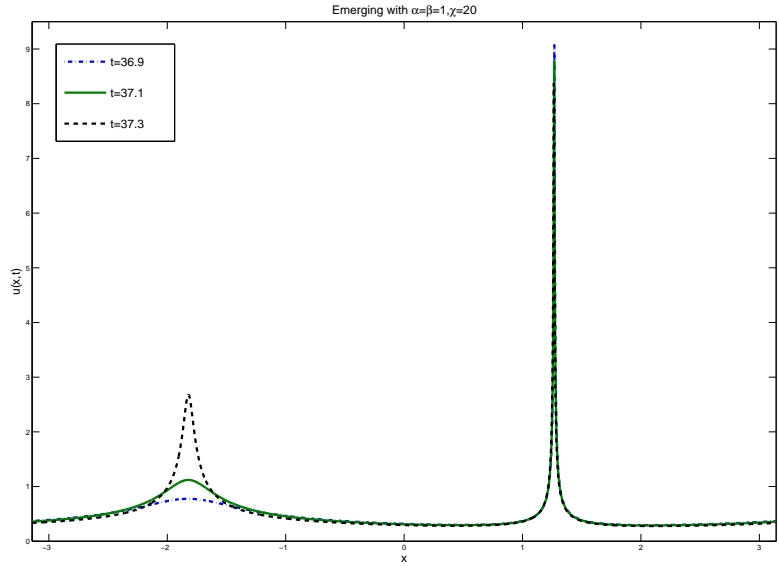}
\caption{Emerging peak.}
\label{emerging}
\end{figure}
\begin{figure}[t!]
    \centering
    \includegraphics[scale=0.28]{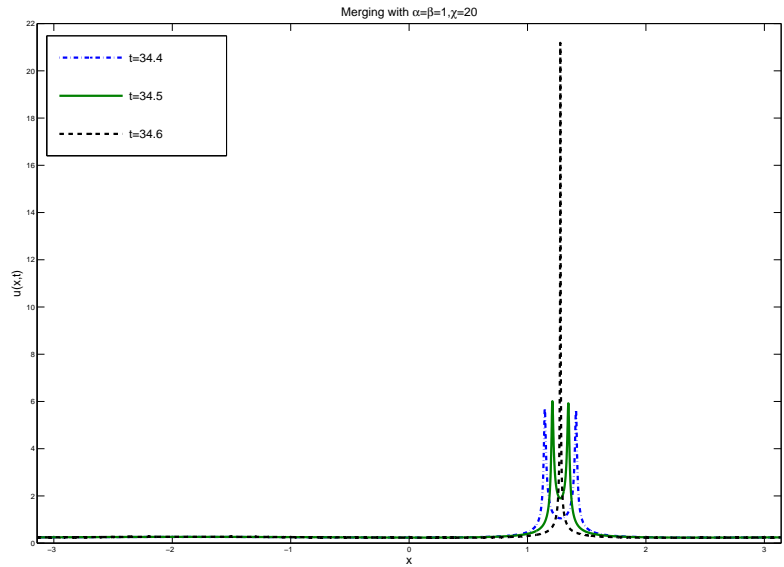}
\caption{Merging peaks.}
\label{merging}
\end{figure}
\begin{figure}[t!]
    \centering
    \includegraphics[scale=0.28]{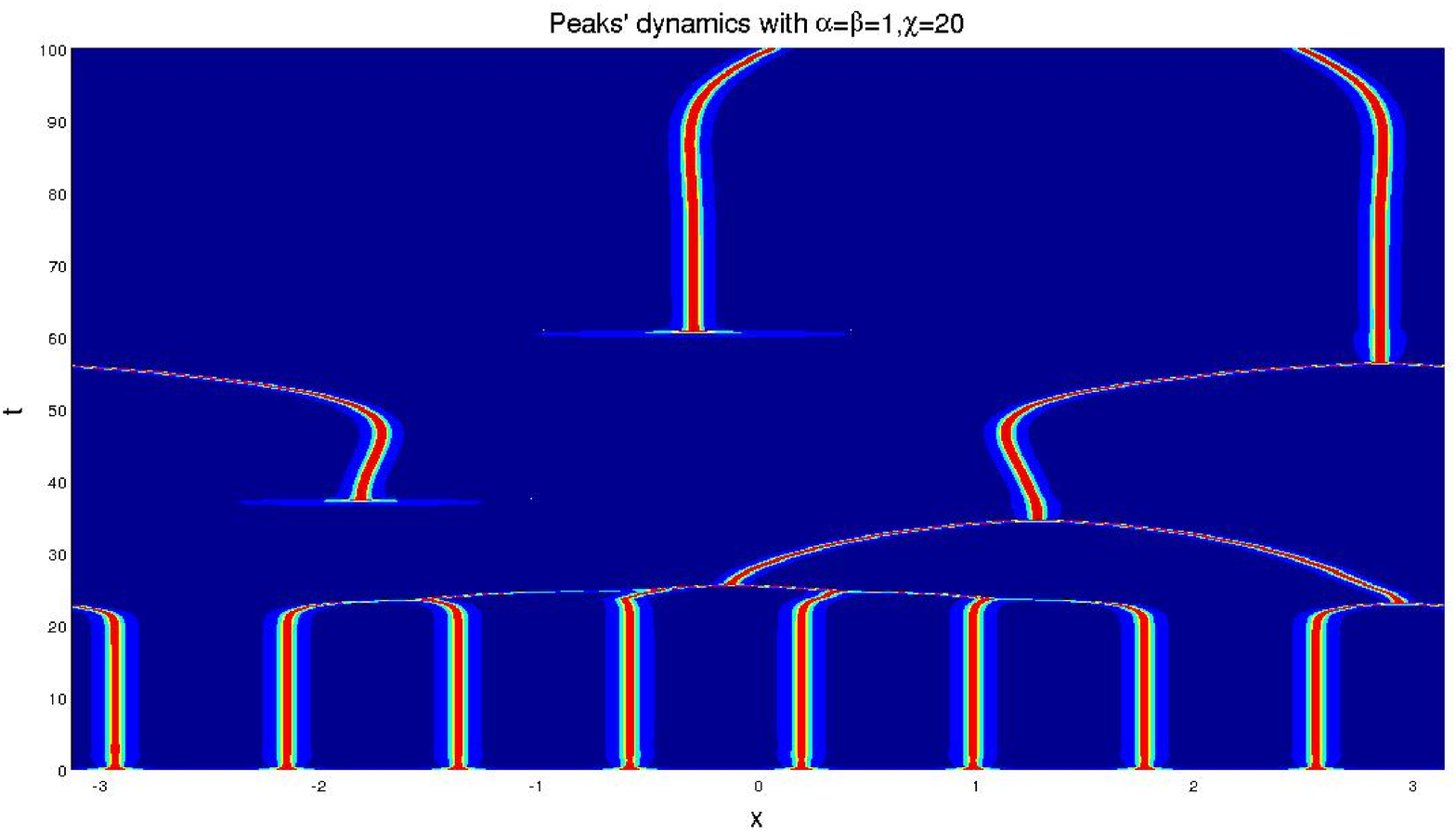}
\caption{Tracking the peaks for the case $\alpha=\beta=1$, $\chi=20$ up to time $T=100$.}
\label{peaksdynamics2}
\end{figure}
We take $N=2^{13}$, $T=30$ and the initial data 
\begin{equation}\label{initialdata}
u_0(x)=1,\; v_0(x)=0.1\sin(8x)+1.
\end{equation}
In order to better understand the role of $\chi$, we approximate the solution for different values $\chi_i\in [5,20]$. The step between our values is
$$
\chi_{i+1}-\chi_i=0.5.
$$
\begin{figure}[t!]
    \centering
    \includegraphics[scale=0.31]{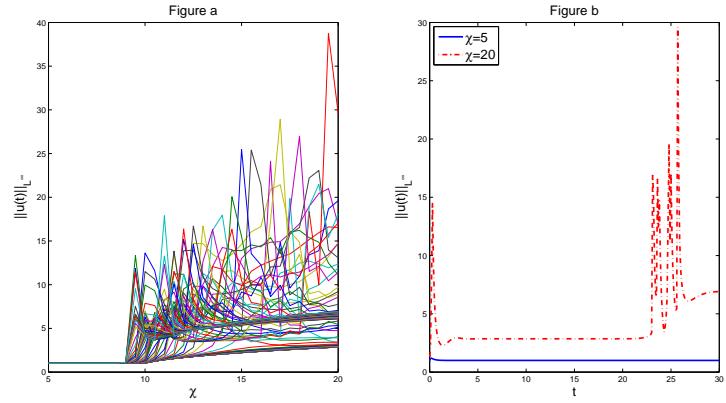}
\caption{Transition to chaos}
\label{transtion}
\end{figure}
The outcome is plotted in Figure \ref{transtion}. In the part \emph{a} of the figure, we plot the solutions corresponding to different values of $\chi$ and times $20\leq t\leq 30$. Notice that every line corresponds to a fixed time $t$. We see that for lower values of $\chi$ the solution tends to the homogeneous steady state, while for large values of $\chi$ the solution develops chaotic behavior. In particular, we can see how a small change in $\chi$ has a substancial impact on the solution at a fixed time $t_i$. We can see also that, for a fixed value $\chi_i$, the solution at different times take very different values.

In part \emph{b} of the same figure, we plot $\|u_N(t)\|_{L^\infty}$ for $\chi=5$ (solid line) and 
$\chi=20$ (dotted line).

\subsubsection{$\alpha=1.5$, $\beta=2$}
Now let us study the case $\alpha=1.5$ and $\beta=2$. Here we take $N=2^{11}$. There are two different cases. One corresponds to $T=100$ and $\chi=20$ and the other to $T=150$ and $\chi=30$. The initial data in both cases is 
$$
u_0(x)=1,\;v_0(x)=2+\text{random}(x),
$$
where $\text{random}(x)$ is a uniformly distributed in $[-0.1,0.1]$ random sample. We recover the same chaotic behaviour with merging and emerging peaks. If we track the peaks, we get the results in Figure \ref{alpha15beta2}.

\begin{figure}[t!]
    \centering
    \includegraphics[scale=0.38]{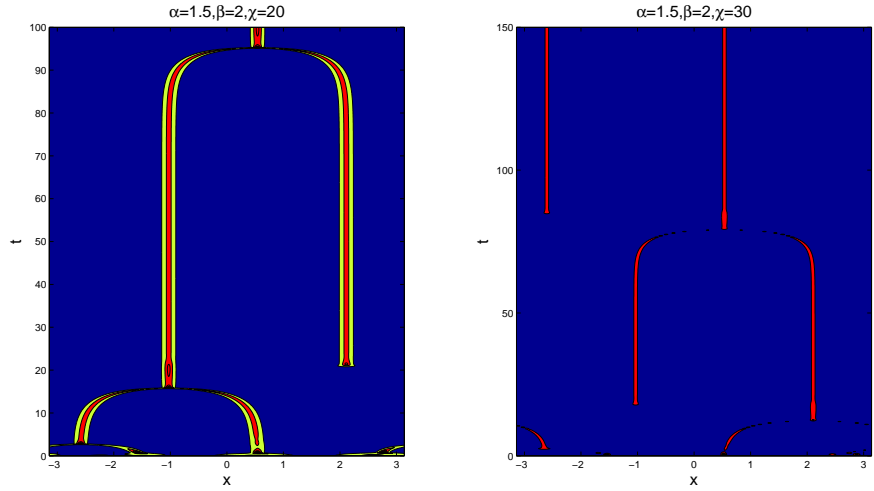}
\caption{Tracking the peaks for $\alpha=1.5$, $\beta=2$ and different values of $\chi$.}
\label{alpha15beta2}
\end{figure}

\subsubsection{$\alpha=1.5$, $\beta=1$}
Now, we consider the case $\alpha=1.5$ and $\beta=1$. Here we take $N=2^{13}$. We simulate the evolution for different values of $\chi\in[16,20]$ up to time $T=20$. The initial data is given by \eqref{initialdata}. In this hyperviscous case, the solutions tend to the homogeneous steady state before the instability appears. Then for time $10<t<20$, $\|u_N(t)\|_{L^\infty}$ grows and several peaks emerge (see Figure \ref{alpha15beta1}).
\begin{figure}[t!]
    \centering
    \includegraphics[scale=0.35]{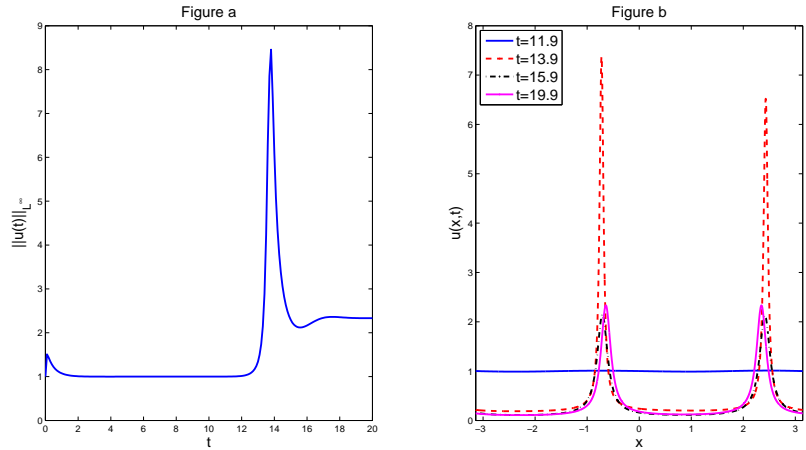}
\caption{a) $\|u_N(t)\|_{L^\infty}$ and b) $u_N(t_i)$ for some $11<t_i<20$ and $\alpha=1.5$, $\beta=1$ and $\chi=20$.}
\label{alpha15beta1}
\end{figure}

\subsubsection{$\alpha=0.5$, $\beta=1$}
In the case $\alpha=0.5$ and $\beta=1$ we take $N=2^{14}$. In this hypoviscous case, we simulate the solution corresponding to two different initial data and two different values of the parameter $\chi$
\begin{equation}\label{initialdata2}
u_0(x)=1,\; v_0(x)=0.1\sin(10x)+1\text{ and }\chi=20,
\end{equation}
\begin{equation}\label{initialdata2b}
u_0(x)=1,\; v_0(x)=0.1\cos(x)\exp(-x^2)+1\text{ and }\chi=10.
\end{equation}
First, we compute the solution corresponding to \eqref{initialdata2} with $\chi=20$ up to time $T=0.11$. This solution appears to have a finite time singularity (see Figure \ref{alpha05beta1}), \emph{i.e.} 
$$
\limsup_{t\rightarrow T_{max}}\|\pax u(t)\|_{L^\infty}=\infty.
$$
Furthermore, if we use least squares to fit a curve with expression
\begin{equation}\label{fitted}
y(t)=\frac{a_1}{(a_2-t)^{a_3}},
\end{equation}
to the numerical solution $\|\pax u_N(t)\|_{L^\infty}$, we obtain the parameters
$$
a_1=0.03389,\;a_2=0.11244,\;a_3=2.14248.
$$
Notice that this evidence of singularity agrees with Theorem \ref{continuation}.
\begin{figure}[t!]
    \centering
    \includegraphics[scale=0.35]{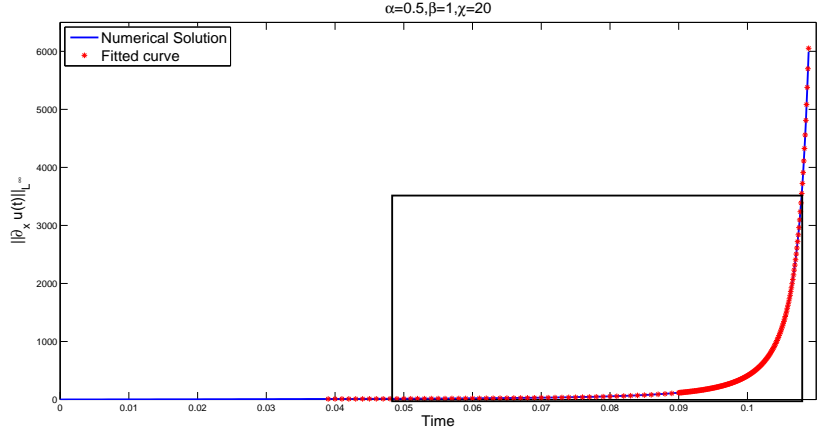}
\caption{$\|\pax u_N(t)\|_{L^\infty}$ for some $0<t_i<0.11$ and $\alpha=0.5$, $\beta=1$ and $\chi=20$. The red points are in the fitted curve. In the box, the points that have been used in the fitting process.}
\label{alpha05beta1}
\end{figure}
On the other hand, if $\chi= 10$, the solution $u(t)$ corresponding to \eqref{initialdata2b} grows in $C^1$ (see Figure \ref{alpha05beta1b}) but a curve like \eqref{fitted} does not approximate the numerical solution $u_N(t)$ well. As a consequence, the value of $\chi$ in the formation of a finite time singularity seems to be crucial.
\begin{figure}[t!]
    \centering
    \includegraphics[scale=0.4]{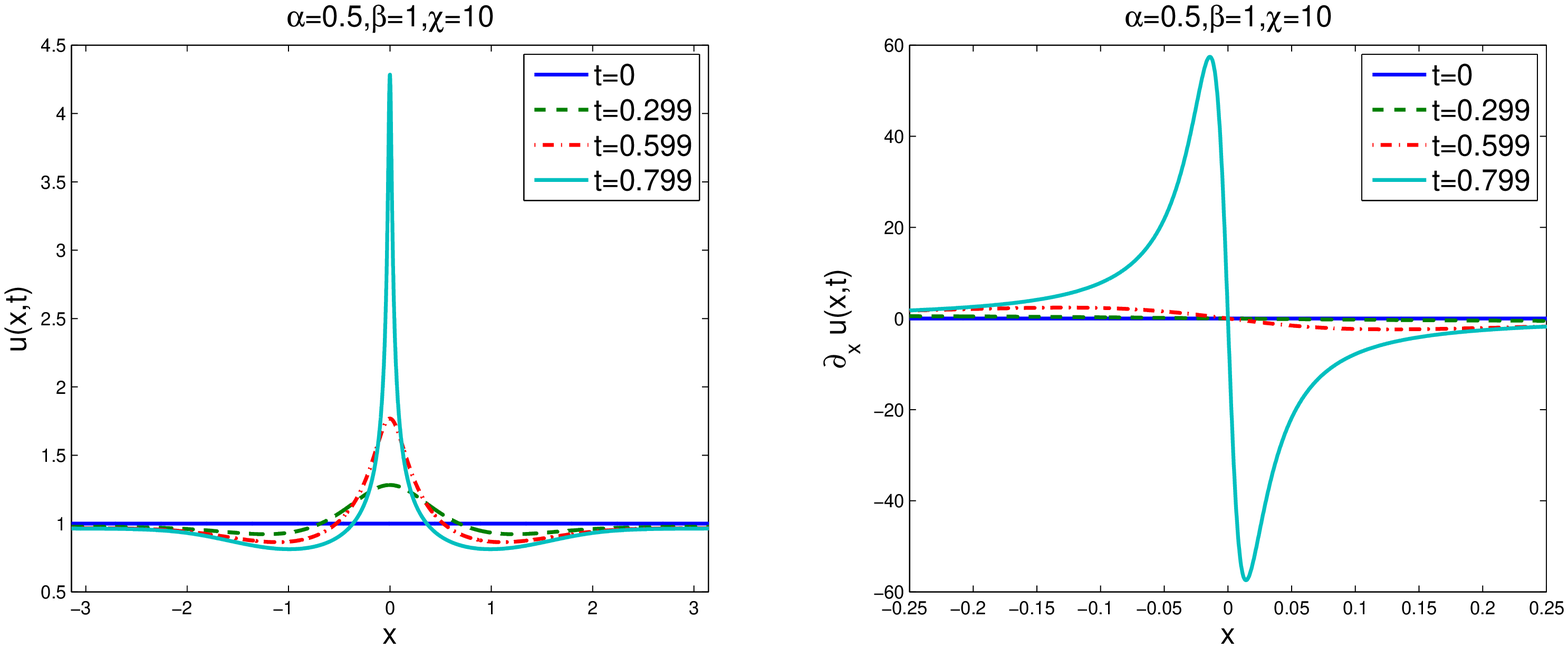}
\caption{a) $u(x,t)$, b) $\pax u_N(t)$ for some $0<t_i<0.9$ and $\alpha=0.5$, $\beta=1$ and $\chi=10$.}
\label{alpha05beta1b}
\end{figure}
Notice that $\chi$ is, roughly speaking, $1/r$. Consequently, the case $r$ \emph{small} presents evidence of singularity while the case $r$ \emph{big} appears to remain smooth for finite time. This is in accordance with \cite{Winkler4}, see also our recent study \cite{BGsuper}.

\appendix
\section{Auxiliary Lemmas}
We state the Kato-Ponce inequality and the Kenig-Ponce-Vega commutator estimate for $[\Lambda^s,F]G=\Lambda^s(FG)-F\Lambda^s G$ and where $\Lambda=\sqrt{-\Delta}$ (see \cite{grafakos2013kato}, \cite{kato1988commutator}, \cite{KenigPonceVega}). 

\begin{lem}\label{lemaaux2}
Let $F,G$ be two smooth functions on $\TT^d$. Then we have the following inequalities:
$$
\|[\Lambda^s,F]G\|_{L^p}\leq C(s,p,p_i)\left(\|F\|_{W^{s,p_1}}\|G\|_{L^{p_2}}+\|G\|_{W^{s-1,p_3}}\|\nabla F\|_{L^{p_4}}\right),
$$
with
$$
\frac{1}{p}=\frac{1}{p_1}+\frac{1}{p_2}=\frac{1}{p_3}+\frac{1}{p_4},\; p,p_1,p_3\in(1,\infty),\; p_2,p_4\in[0,\infty],s>0.
$$
and 
$$
\|\Lambda^s(FG)\|_{L^p}\leq \frac{C_{KP}(s,p,p_i)}{2}\left(\|\Lambda^s F\|_{L^{p_1}}\|G\|_{L^{p_2}}\right.\\
\left.+\|\Lambda^s G\|_{L^{p_3}}\|F\|_{L^{p_4}}\right),
$$
with
$$
\frac{1}{p}=\frac{1}{p_1}+\frac{1}{p_2}=\frac{1}{p_3}+\frac{1}{p_4}, \frac{1}{2}<p<\infty,1<p_i\leq\infty, s>\max\{0,d/p-d\}.
$$
\end{lem}
\begin{remark}In particular, we are using the notation
$$
C_{KPV}(\alpha)=C\left(\frac{\alpha}{2},2,\frac{2}{\alpha-1},2+\frac{2\alpha-2}{2-\alpha},\infty,2\right).
$$
\end{remark}
We require the following uniform Gronwall lemma (see \cite{temambook}).
\begin{lem}\label{UGL}
Suppose that $g$, $h$, $y$ are non-negative, locally integrable functions on $(0,\infty)$ and ${dy}/{dt}$ is locally integrable. If
there are positive constants $a_1$, $a_2$, $a_3$, $b$ such that
$$
\frac{dy}{dt}\leq gy+h,\quad
\int_t^{t+b} g(s)ds\leq a_1,\;\int_t^{t+b} h(s)ds\leq a_2,\;\int_t^{t+b} y(s)ds\leq a_3
$$
for $t \ge 0$, then
$$
y(t+b)\leq \left(\frac{a_3}{b}+a_2\right)e^{a_1}.
$$
\end{lem}

The last Lemma studies the number of critical points of an analytic function (compare Gruji\'c \cite{AnalyticityKuramotoGrujic}).
\begin{lem}\label{grujic}
Let $w>0$, and let $u$ be analytic in the neighbourhood of $\{z: |\Im z|\leq w\}$ and $2\pi$-periodic in the $x$-direction. Then, for any $\epsilon>0$ holds $\TT=I_\epsilon\cup R_\epsilon$, where $I_\epsilon$ is an union of at most $[\frac{4\pi}{w}]$ intervals open in $\TT$, and
\begin{itemize}
\item $|\pax u(x)| \leq \epsilon, \text{ for all }x\in I_\epsilon,$
\item $\text{card}\{x \in R_\epsilon : \pax u(x)=0\}\leq
\frac{2}{\log 2}\frac{2\pi}{w}\log\left(\frac{\max_{|\Im z|\leq w}|\pax u(z)|}{\epsilon}\right).$
\end{itemize}
\end{lem}

\section{Explicit expressions for the constants}\label{AppB}
Here we collect explicit expressions for the constants that we use in Theorem \ref{globalpp1abs} and Corollaries \ref{oscillationsattractor}, \ref{oscillationsattractor2}.
We define the radius of the absorbing set in $L^2$ as
\begin{multline}\label{SL2}
S(L^2)=e^r\bigg{[}3\mathcal{N}+\frac{C_{KP}(\alpha)^2\mathcal{N}^3}{\mu}2C_{FS}(\beta,\alpha,\lambda,\nu)\\
+\frac{(C_{KP}(\alpha)C_{GN}(\alpha))^{\frac{2\alpha+2}{\alpha-1}}4\mathcal{N}^{2}}{\mu}\left(2\mathcal{N}C_{FS}(\beta,\alpha,\lambda,\nu)\right)^{\frac{\alpha+1}{\alpha-1}}\bigg{]}.
\end{multline}
Let
\begin{multline}\label{SI}
I= r+\frac{r^2C^2_{SE}(\alpha)}{\mu}6\mathcal{N}+\frac{\mathcal{N}}{\nu}\left(\frac{3}{\nu}+2C_{FS}(\beta,\alpha,\lambda,\nu)\right)\\
\times\left(\frac{1}{2}+\frac{2C^2_{SE}(\alpha)+(C_I(\alpha))^2}{\mu}\right.\\
\times\left.\frac{\left[C_{KPV}(\alpha)(C^3_{SE}(\alpha)C^1_{SE}(\alpha)+C^4_{SE}(\alpha))\right]^2}{\mu}\right),
\end{multline}

the radius of the absorbing set in higher norms is given by
\begin{equation}\label{SHalpha2}
S(H^{\alpha/2})\leq S(L^2)+
S(\dot{H}^{\alpha/2})=S(L^2)\left(1+\frac{2(1+e^{-r})}{\mu}e^{2I}\right).
\end{equation}
We denote
\begin{equation}\label{K1}
\mathcal{K}_1(\alpha,\beta,\mu,\nu,r,\lambda)=1+(\mu+2+2r+\nu+\lambda+1)C_{SE}^2(\alpha)+\mathcal{C}_1+
2\mathcal{C}_2+\mathcal{C}_3+\mathcal{C}_4,
\end{equation}
\begin{equation}\label{K2}
\mathcal{K}_2(\alpha,\beta,\mu,\nu,r,\lambda)=1+(1+\mu+2+2r+\nu+\lambda+1)C_{SE}^2(1.1)+\mathcal{C}_5+\mathcal{C}_6,
\end{equation}
\begin{equation}\label{K}
\mathcal{K}(\alpha,\beta,\mu,\nu,r,\lambda)=\left\{\begin{array}{ll}
\mathcal{K}_1\text{ if }\alpha,\beta>1,\\
\mathcal{K}_2\text{ if }\min\{\alpha,\beta\}=1,
\end{array}\right.
\end{equation}
where 
\begin{equation}\label{C1}
\mathcal{C}_1=\max_{\xi\in\RR^+}4\omega\xi-\frac{\mu}{2}\xi^{\alpha} = 4\omega\left(\frac{8\omega}{\mu \alpha}\right)^{\frac{1}{\alpha-1}}-\frac{\mu}{2}\left(\frac{8\omega}{\mu \alpha}\right)^{\frac{\alpha}{\alpha-1}},
\end{equation}
\begin{equation}\label{C2}
\mathcal{C}_2 =\max_{\xi\in\RR^+}3\omega\xi-\nu\xi^{\beta} = 3\omega\left(\frac{3\omega}{\nu \beta}\right)^{\frac{1}{\beta-1}}-\nu\left(\frac{3\omega}{\nu \beta}\right)^{\frac{\beta}{\beta-1}},
\end{equation}
\begin{multline}\label{C3}
\mathcal{C}_3=2\bigg{[}\frac{(C_{SI}(\alpha))^2}{\omega}\frac{\alpha-1}{\alpha}\left(\frac{\omega^2 \alpha}{(C_{SI}(\alpha))^2}\right)^{-1/(\alpha-1)} +2\omega\\
+1.25 +2r+(17.5C_{SE}^2(\alpha))^2+\frac{(9rC_{SE}^2(\alpha))^2}{2}\bigg{]},
\end{multline}
\begin{equation}\label{C4}
\mathcal{C}_4=1.5+\frac{2(C_{KPV}^2C_{SE}^2(\alpha))^2}{\mu},
\end{equation}
\begin{equation}\label{C5}
\mathcal{C}_5=2\bigg{[}\frac{\mu}{4}+1.25
+2r+(17.5C_{SE}^2(1.1))^2+\frac{(9rC_{SE}^2(1.1))^2}{2}\bigg{]},
\end{equation}
\begin{equation}\label{C6}
\mathcal{C}_6=1.5+\frac{2(C_{KPV}^2C_{SE}^2(1.1))^2}{\mu}.
\end{equation}

\subsection*{Acknowledgments} RGB is partially supported by the grant MTM2014-59488-P from the former Ministerio de Econom\'ia y Competitividad (MINECO, Spain). We thank deeply an anonymous Referee for remarks that greatly improved the first version of this paper. 

\bibliographystyle{abbrv}

\end{document}